\documentclass[12pt]{amsart}
\usepackage{amsmath, cite, tikz}
\usepackage[colorlinks]{hyperref}
\usepackage[capitalize]{cleveref}
\usepackage[a4paper]{geometry}
\usetikzlibrary{calc, decorations.pathmorphing}

\newtheorem{thm}{Theorem}[section]
\newtheorem{conj}[thm]{Conjecture}
\newtheorem{lem}[thm]{Lemma}
\newtheorem{cor}[thm]{Corollary}

\theoremstyle{definition}
\newtheorem{case}{Case}
\newtheorem{subcase}{Subcase}
\newtheorem{clm}{Claim}[case]

\makeatletter \@addtoreset{equation}{section} \makeatother

\def\={\;=\;}
\def\<{\;<\;}
\def\>{\;>\;}
\def\lle{\;\le\;}
\def\gge{\;\ge\;}
\def\ss{\;\subset\;}
\def\sse{\;\subseteq\;}
\def\fl#1{\lfloor{#1}\rfloor}

\def\cl#1{\lceil{#1}\rceil}
\def\Bcl#1{\Bigl\lceil\,{#1}\,\Bigr\rceil}
\def\bggcl#1{\biggl\lceil\,{#1}\,\biggr\rceil}

\def\Bg#1{\Big({#1}\Bigr)}
\def\bgg#1{\biggl({#1}\biggr)}
\def\nf{\frac{n}{4}}
\def\nt{\frac{n}{2}}
\def\M{\mathcal{M}}
\def\rmand{\quad\hbox{ and }\quad}

\crefname{def}{Definition}{Definitions}
\Crefname{def}{Definition}{Definitions}
\creflabelformat{def}{~\upshape(#2#1#3)}
\crefname{clm}{Claim}{Claims}
\Crefname{clm}{Claim}{Claims}
\creflabelformat{clm}{~\upshape#2#1#3}
\crefname{ineq}{Ineq.}{Ineqs.}
\Crefname{ineq}{Inequality}{Inequalities}
\creflabelformat{ineq}{~\upshape(#2#1#3)}
\crefname{rl}{Relation}{Relations}
\crefname{rl}{Relation}{Relations}
\creflabelformat{rl}{~\upshape(#2#1#3)}
\crefname{case}{Case}{Cases}
\crefname{case}{Case}{Cases}
\creflabelformat{case}{~\upshape#2#1#3}

\def\mathllapinternal#1#2{\llap{$\mathsurround=0pt#1{#2}$}}
\def\mathrlapinternal#1#2{\rlap{$\mathsurround=0pt#1{#2}$}}
\def\mathllap{\mathpalette\mathllapinternal}
\def\mathrlap{\mathpalette\mathrlapinternal}

\allowdisplaybreaks

\begin{document}

\title[Perfect matchings of semi-regular graphs]
{The maximum number of perfect matchings of semi-regular graphs}

\author[H. Lu]{Hongliang Lu}
\address{School of Mathematics and Statistics, Xi'an Jiaotong university, 710049 Xi'an, P.\ R.\ China}
\email{luhongliang@mail.xjtu.edu.cn}
%\thanks{H. Lu is supported by the National Natural Science Foundation of China, No.\ 11471257.}

\author[D.G.L. Wang]{David G.L. Wang$^{\dag\ddag}$}
\address{
$^\dag$School of Mathematics and Statistics, Beijing Institute of Technology, 102488 Beijing, P.\ R.\ China\\
$^\ddag$Beijing Key Laboratory on MCAACI, Beijing Institute of Technology, 102488 Beijing, P.\ R.\ China}
\email{glw@bit.edu.cn}
%\thanks{D.G.L. Wang is supported by the Beijing Institute of Technology Research Fund Program for Young Scholars.}

\date{}

\keywords{factorization, Hamiltonian graph, perfect matching, regular graph}

\maketitle

\begin{abstract}
Let $n\ge 34$ be an even integer, and $D_n=2\cl{n/4}-1$.
In this paper, we prove that every $\{D_n,\,D_n+1\}$-graph of order~$n$ 
contains~$\cl{n/4}$ disjoint perfect matchings.
This result is sharp in the sense that 
(i) there exists a $\{D_n,\,D_n+1\}$-graph containing exactly $\cl{n/4}$ disjoint perfect matchings, and that
(ii) there exists a $\{D_n-1,\,D_n\}$-graph without perfect matchings for each $n$.
As a consequence, for any integer $D\ge D_n$, 
every $\{D,\,D+1\}$-graph of order~$n$ contains $\cl{(D+1)/2}$ disjoint perfect matchings.
This extends Csaba et~al.'s breathe-taking result that 
every $D$-regular graph of sufficiently large order is $1$-factorizable,
generalizes Zhang and Zhu's result that every $D_n$-regular graph of order~$n$ 
contains $\cl{n/4}$ disjoint perfect matchings,
and improves Hou's result that for all $k\ge n/2$, every $\{k,\,k+1\}$-graph of order~$n$
contains $(\fl{n/3}+1+k-n/2)$ disjoint perfect matchings.
\end{abstract}

\section{Introduction}

Vizing's theorem~\cite{Viz64} states that the edge-chromatic number of any graph 
is equal to or one more than the maximum degree of the same graph.
The problem of determining the precise value of the edge-chromatics number 
for an arbitrary graph is NP-complete; see Holyer~\cite{Hol81}.
For any regular graph, 
its edge-chromatic number equals its maximum degree if and only if the graph is a $1$-factorizable,
i.e., its edge set can be decomposed into perfect matchings.
Here is the famous $1$-factorization \cref{conj:1ftrz}.

\begin{conj}[The $1$-factorization conjecture]\label{conj:1ftrz}
Every regular graph of even order with sufficiently high degree is $1$-factorizable. 
\end{conj}

It is considered to be Chetwynd and Hilton who first stated that \cref{conj:1ftrz} explicitly,
though they~\cite{CH85} claimed that the conjecture had been discussed in the 1950s, according to Dirac.
They showed that every graph of even order $n$ with minimum degree at least $6n/7$ is $1$-factorizable.
This bound was improved to \hbox{$(\sqrt{7}-1)n/2$} later, 
by the same authors~\cite{CH89}, and Niessen and Volkmann~\cite{NV90} independently.
Plantholt and Tipnis~\cite{PT91} further generalized this bound to multigraphs.
Focusing on $k$-regular graphs with $k\ge n/2$,
Hilton~\cite{Hil85} managed to peel off~$\fl{k/3}$ disjoint $1$-factors depending on the graph degree.
Remarkably, Zhang and Zhu~\cite{ZZ92} improved the bound $\fl{n/3}$ to a sharp one.

\begin{thm}[Zhang and Zhu]\label{thm:ZZ}
Any $k$-regular graph of even order $n$ such that $k\ge n/2$ contains at least $\fl{k/2}$ disjoint perfect matchings.
\end{thm}

Very recently, Csaba et al.~\cite{CKLOT14X} obtained the following astonishing breakthrough.
Let $n$ be an even integer and define
\begin{equation}\label[def]{def:D}
D_n\=2\Bcl{\nf}-1\=\begin{cases}
\displaystyle \nt-1,&\text{if $n\equiv 0\pmod4$};\\[8pt]
\displaystyle \nt,&\text{if $n\equiv 2\pmod4$}.
\end{cases}
\end{equation}

\begin{thm}[Csaba et al.]\label{thm:CKLOT}
Let $n$ be a sufficient large even integer, and let $D\ge D_n$. 
Then every $D$-regular graph $G$ of order $n$ is $1$-factorizable. In other words,
the edge-chromatic number~$\chi'(G)$ equals the degree~$D$.
\end{thm}

For any set $S$ of non-negative integers, we call a graph {\em $S$-regular}, or an {\em $S$-graph},
if the degree of every its vertex belongs to $S$.
Following Akiyama and Kano's book~\cite[Section~5.2]{AK11B}, 
we call an $S$-graph {\em semi-regular} if the set $S$ consists of two adjacent integers.
Yet another perspective, 
Hou~\cite{Hou11} generalized Hilton's result to semi-regular graphs.

\begin{thm}[Hou]\label{thm:Hou}
Every $\{k,\,k+1\}$-graph of even order $n\le 2k$ contains at least $(\fl{n/3}+1+k-n/2)$ disjoint perfect matchings.
\end{thm}

In this paper, we consider the $1$-factorization problem of semi-regular graphs.
We improve Hou's \cref{thm:Hou} to the sharp result that 
every $\{D_n,\,D_n+1\}$-graph of even order $n\ge34$ contains $\cl{n/4}$ disjoint perfect matchings; see \cref{thm:main}.
This result generalizes Zhang and Zhu's \cref{thm:ZZ}
and extends Csaba et al.'s \cref{thm:CKLOT}.

\section{Preliminary}

In this paper, we consider finite undirected simple graphs without loops or multiple edges. 
The number of vertices in a graph~$G$ is said to be the order of~$G$, denoted~$|G|$.
As usual, we denote
the neighbor set of a vertex subset~$W$ of~$G$ by~$N_G(W)$, or simply~$N(W)$ if there is no confusion.
One of the earliest corner-stones in the matching theory is Hall's theorem~\cite{Hal35}.

\begin{thm}[Hall]\label{thm:Hall}
Let $G=(X,Y)$ be a bipartite graph. Then $G$ has a matching covering~$X$ if and only if
$|W|\le |N(W)|$ for every subset $W$ of $X$.
\end{thm}

The famous Tutte's theorem~\cite{Tutte} states that
a graph $G$ has a perfect matching if and only if for any vertex subset $S$,
the number of odd components of the graph $G-S$ is at most the order $|S|$.
In this paper, we will use the following stronger version of Tutte's theorem, 
see Lov\'asz and Plummer's book~\cite[Exercise 3.3.18~(b)]{LP86B}.
A graph~$G$ is said to be {\em factor-critical} if the subgraph $G-u$ has a perfect matching for every vertex~$v$.

\begin{thm}\label{thm:Berge}
Let $G$ be a graph without perfect matchings.
Then~$G$ has a vertex subset~$S$ such that every component of the subgraph $G-S$ is factor-critical,
and that the number $o(G-S)$ of components of the subgraph $G-S$ satisfies 
\[
o(G-S)\equiv |S|\pmod{2}
\rmand
o(G-S)\gge |S|+2.
\]
\end{thm}

We also need some known results judging the graph structure with aid of the minimum degree.
A graph that contains a Hamiltonian cycle is called {\em Hamiltonian}.
Next is a classical criterion for graph Hamiltonicity due to Dirac~\cite{Dir52}.

\begin{thm}[Dirac]\label{thm:Dirac}
Every graph with minimum degree at least half of its order is Hamiltonian.
\end{thm}

A graph is said to be {\em Hamiltonian-connected} 
if it contains a Hamiltonian path between every two distinct vertices.
Ore~\cite{Ore63} discovered a criterion for this stronger property.

\begin{thm}[Ore]\label{thm:Ore}
Let $G$ be a $2$-connected graph. 
Suppose that the degree sum of every two non-adjacent vertices of $G$ is larger than the order $|G|$.
Then $G$ is Hamiltonian-connected.
\end{thm}

Note that every Hamiltonian graph is $2$-connected.
With aid of Dirac's \cref{thm:Dirac}, the following corollary of \cref{thm:Ore} holds true.
See also~\cite[10.24]{Lov79B}.

\begin{cor}\label{cor:Ore}
Any graph $G$ of minimum degree more than $|G|/2$ is Hamiltonian-connected.
\end{cor}

A graph~$G$ is said to be {\em bi-critical} 
if the subgraph $G-u-v$ has a perfect matching for every two distinct vertices~$u$ and~$v$. 
The minimum degree, as expectable, can also be used to determine the bi-criticality of graphs.

\begin{lem}[Plummer, \cite{Plu90}]\label{lem:Plummer}
Let $G$ be a connected graph of even order $n$.
If the minimum degree of $G$ is larger than $n/2$, 
then the graph $G$ is bi-critical.
\end{lem}

Let us give an overview of notion and notations that we need in the sequel. 
For any vertex subset~$S$ of~$V$, 
we denote by $G[S]$ the subgraph of~$G$ induced by~$S$, and write $G-S=G[V(G)-S]$.
For a graph $G$ and an edge set $\tilde{E}$,
we denote by $G\cup \tilde{E}$ the graph with vertex set $V(G)\cup V(\tilde{E})$
and edge set $E(G)\cup \tilde{E}$.

For any vertex subsets $X$ and $Y$ of a graph $G$, we denote by $E_G(X,Y)$ the set of edges with one end in $X$ and the other end in $Y$. It is clear that $E_G(X,Y)=E_G(Y,X)$. Denote $e_G(X,Y)=|E_G(X,Y)|$. 
As usual, we use the notation 
\[
\partial_G X\=E_G(X,\,V(G)- X).
\]
The degree of a vertex~$v$ in a graph~$G$ is denoted by $\deg_G(v)$.
The minimum degree of vertices of a vertex set $X$ in a graph $G$ is denoted by $\delta_G(X)$.
As usual, we denote $\delta(G)=\delta_G(V(G))$.
When the symbol $X$ or $Y$ denotes a subgraph of $G$, 
we use the same notation $E_G(X,Y)$ to denote the edge set $E_G(V(X),\,V(Y))$,
and use the similar convention $\delta_G(X)=\delta_G(V(X))$.

\section{Main Result}

\Cref{lem} will be of considerable help in the proof of \cref{thm:main}.

\begin{lem}\label{lem}\label{lem:Hall}
Let $d,k,s$ be integers such that $d\ge (s+k)/2+1$ and $d\ge k+1$.
Let $G'=(S,U)$ be a bipartite graph with part orders $|S|=s$ and $|U|=s+1$. 
Suppose that the minimum degree~$\delta_{G'}(U)$ is at least~$d$,
and that every vertex in the part~$S$ has degree at most $(d+2)$,
with at most one vertex in~$S$ having degree $(d+2)$.
Then for any vertex subset $S'\subset S$ of order $k$ and for any vertex subset $U'\subset U$ of order $(k+1)$,
the graph $G'-S'-U'$ has a perfect matching.
\end{lem}

\begin{proof}
By contradiction, suppose that there exist subsets $S'\subset S$ and $U'\subset U$ 
such that the subgraph $H=G'-S'-U'$ has no perfect matchings. 
By Hall's \cref{thm:Hall}, there exists a vertex set $T \subseteq U-U'$ such that 
\begin{equation}\label[ineq]{NT<T}
|N_{H}(T)|\lle |T|-1.
\end{equation} 
See \cref{fig:lem:Hall}. Denote $p=|N_{H}(T)|$.
By using the hand-shaking theorem, we have
\begin{equation}\label{handshaking}
\sum_{u\in U}\deg_{G'}(u)
\=\sum_{v\in S}\deg_{G'}(v)
\=\sum_{\mathrlap{v\in N_{H}(T)\cup S'}}\deg_{G'}(v)
\;+\;\sum_{\mathrlap{v\in S-N_{H}(T)-S'}}\deg_{G'}(v).
%\=\sum_{x\in N_{H}(T)\cup S'}\deg_G x\;+\;\sum_{x\in S- S'-N_{H}(T)}\deg_G x. 
\end{equation}
We shall estimate the three summations on both sides of \cref{handshaking} individually.

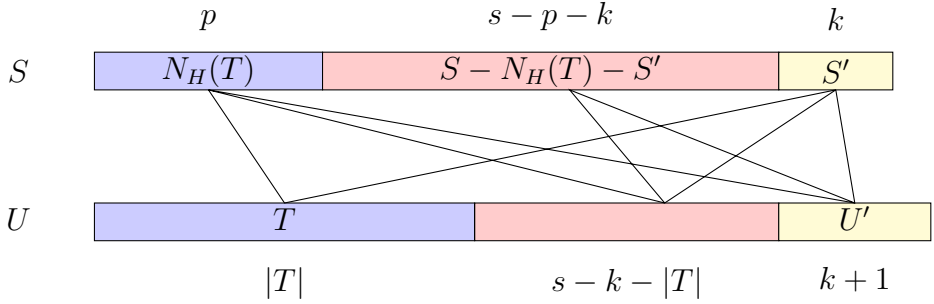
\begin{figure}[htbp]
\centering
\begin{tikzpicture}
\coordinate (U) at (-1, 0);
\coordinate (S) at (-1, 2.25);

\coordinate (T) at (2.5, 0);
\coordinate (U') at (10, 0);
\coordinate (N) at (1.5, 2.25);
\coordinate (S') at (9.75, 2.25);
\coordinate (S'') at (6, 2.25);
\coordinate (U'') at (7, 0);

\coordinate (U'p) at (10, .5);
\coordinate (Np) at (1.5, 2);
\coordinate (S''p) at (6.25, 2);
\coordinate (S'p) at (9.75, 2);
\coordinate (Tp) at (2.5, .5);
\coordinate (U''p) at (7.5, .5);

\draw[fill=blue!20] 
(0, 0) rectangle (5, .5)
(0, 2) rectangle (3, 2.5);
\draw[fill=red!20] 
(5, 0) rectangle (9, .5)
(3, 2) rectangle (9, 2.5);
\draw[fill=yellow!20] 
(9, 0) rectangle (11, .5)
(9, 2) rectangle (10.5, 2.5);
\draw
(5, 0) -- (5, .5)
(Np) -- (Tp) -- (S'p)
(U'p) -- (S''p) -- (U''p)
(Np) -- (U''p) -- (S'p)
(Np) -- (U'p) -- (S'p);
\draw
(S)node[]{$S$}
(U)node[above]{$U$}
(T)node[above]{$T$}
(U')node[above]{$U'$}
(N)node[]{$N_H(T)$}
(S')node[]{$S'$}
(S'')node[]{$S-N_H(T)-S'$}
(N)node[above=4mm]{$p$}
(S'')node[above=4mm]{$s-p-k$}
(S')node[above=4mm]{$k$}
(T)node[below=2mm]{$|T|$}
(U')node[below=2mm]{$k+1$}
(U'')node[below=2mm]{$s-k-|T|$}
;
\end{tikzpicture}
\caption{The graph $G'$.}\label{fig:lem:Hall}
\end{figure}

From the premise that every vertex in the part~$U$ has degree at least $d$, 
we infer that
\[
\sum_{u\in U}\deg_{G'}(u)\gge d\cdot|U|\= d\,(s+1).
\]
From the premise that every vertex in the part~$S$ has degree at most $(d+2)$, 
with at most one vertex having degree $(d+2)$, 
we deduce that
\[
\sum_{\mathrlap{v\in N_{H}(T)\cup S'}}\deg_{G'}(v)
\lle (d+2)+(d+1)\cdot (|N_{H}(T)\cup S'|-1)
\=1+(d+1)(p+k).
\]
Note that the neighbors of all vertices in the set $S-N_{H}(T)-S'$ are in the set $U-T$. 
Therefore, with the aid of \cref{NT<T},
we derive that 
\begin{align*}
\sum_{\mathllap{v\in S-N_{H}(T)-S'}}\deg_{G'}(v)
&\lle |S- N_{H}(T)- S'|\cdot|U-T|\\[-5pt]
&\= (s-p-k)(s+1-|T|)
\lle (s-p-k)(s-p). 
\end{align*}
%\begin{align}
%\sum_{\mathllap{x\in S-N_{H}(T)-S'}}\deg_G x
%&\lle |S- N_{H}(T)- S'|\cdot|U-T|\notag\\[-5pt]
%&\= (s-p-k)(s+1-|T|)
%\lle (s-p-k)(s-p). \label[ineq]{deg:S''}
%\end{align}
Combining the above three inequalities 
%\cref{deg:U,deg:NT,deg:S''} 
with \cref{handshaking}, we obtain that
\begin{equation}\label[ineq]{ineq:handshaking}
d\,(s+1)\lle 1+(d+1)(p+k)+(s-p-k)(s-p).
\end{equation}

To deal with \cref{ineq:handshaking}, we first figure out the domain of~$p$.
On the one hand, we have $T\ne\emptyset$ in virtue of \cref{NT<T}.
From the premise, every vertex in the set~$T$ has at least $d$ neighbors.
Thus $|N_{G'}(T)|\ge d$ and thereby
\[
|N_{H}(T)|\gge |N_{G'}(T)|-|S'|\gge d-k.
\]
On the other hand, from definition, we have $T\subseteq U-U'$. 
Together with \cref{NT<T}, we obtain 
\[
p\lle |T|-1
\lle |U-U'|-1\=(s+1)-(k+1)-1\=s-k-1.
\]
Combining the above two inequalities, we find the domain 
\[
d-k\lle p\lle s-k-1.
\]

In view of the premises $d\ge (s+k)/2+1$ and $d\ge k+1$, and the above domain of~$p$,
it is elementary to derive that the right hand side of \cref{ineq:handshaking},
considered as a quadratic function in the variable~$p$, attains its maximum at the value $p=s-k-1$.
Therefore, we can substitute $p=s-k-1$ into \cref{ineq:handshaking}, which gives
\[
d\,(s+1)\lle 1+(d+1)(s-1)+(k+1),
\]
contradicting the premise $d\ge (s+k)/2+1$.
This completes the proof.  
\end{proof}

\begin{lem}\label{lem:CC:s=0}
Let $H$ be a graph with minimum degree at least $\cl{n/4}$,
consisting of factor-critical components~$C_1$ and~$C_2$ with $|C_1|\le |C_2|$.
Let $M$ be a perfect matching of the complementary graph of~$H$.
Let $M'$ be a perfect matching of the graph $H\cup M$ such that 
the graph $(H\cup M)-M'$ consists of factor-critical components~$C_1'$ and~$C_2'$ with $|C_1'|\le |C_2'|$.
Suppose that 
\begin{equation}\label[ineq]{cond:lem}
E_M(C_1,C_2)-M'\ne\emptyset.
\end{equation}
Then we have $V(C_1')\subset V(C_2)$. In other words, we have 
\[
V(C_1)\cap V(C_1')\=\emptyset
\rmand
V(C_2)\cap V(C_2')\;\ne\;\emptyset.
\]
\end{lem}

\begin{proof}
Denote $H'=(H\cup M)-M'$. 
Since the minimum degree $\delta(H)\ge\cl{n/4}$, we find
\begin{equation}\label[ineq]{lb:c1:lem}
|C_1|\gge \nf+1.
\end{equation}
For $i,j\in[2]$, we denote
\[
V_{ij}\=V(C_i)\cap V(C_j').
\]
Then the desired results are $V_{11}=\emptyset$ and $V_{22}\neq\emptyset$. 
See \cref{fig:CC}.
In the colorful version, one may see that the component $C_1'$ is in red,
while the component $C_2'$ is in blue.

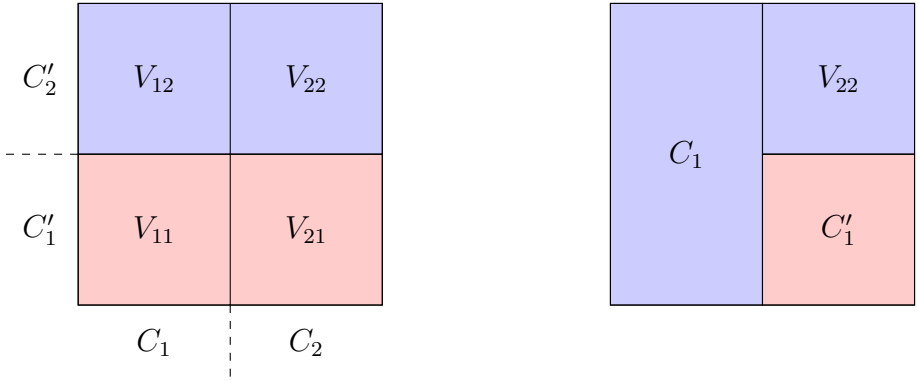
\begin{figure}[htbp]
\centering
\begin{tikzpicture}
\begin{scope}
\draw[fill=blue!20] 
(0,2) rectangle (4,4);
\draw[fill=red!20] 
(0,0) rectangle (4,2);

\coordinate (C1) at (1, -.5);
\coordinate (C2) at (3, -.5);
\coordinate (C1') at (-.5, 1);
\coordinate (C2') at (-.5, 3);
\coordinate (V11) at (1, 1);
\coordinate (V12) at (1, 3);
\coordinate (V21) at (3, 1);
\coordinate (V22) at (3, 3);
\draw
(0, 0) rectangle (4, 4)
(2, 0) -- (2, 4)
(0, 2) -- (4, 2)
;
\draw[dashed]
(0, 2) -- +(-1, 0)
(2, 0) -- +(0, -1)
;
\draw
(C1)node[]{$C_1$}
(C2)node[]{$C_2$}
(C1')node[]{$C_1'$}
(C2')node[]{$C_2'$}
(V11)node[]{$V_{11}$}
(V12)node[]{$V_{12}$}
(V21)node[]{$V_{21}$}
(V22)node[]{$V_{22}$}
;
\end{scope}
%:
\begin{scope}[xshift =7cm]
\draw[fill=blue!20] 
(0,0) rectangle (2,4)
(2,2) rectangle (4,4);
\draw[fill=red!20] 
(2,0) rectangle (4,2);

\coordinate (C1) at (1,2);
\coordinate (C2) at (3, -.5);
\coordinate (C1') at (3, 1);
\coordinate (C2') at (-.5, 3);
\coordinate (V22) at (3, 3);
\draw
(0, 0) rectangle (4, 4)
(2, 0) -- (2, 4)
(2, 2) -- (4, 2)
;
\draw
(C1)node[]{$C_1$}
(C1')node[]{$C_1'$}
(V22)node[]{$V_{22}$}
;
\end{scope}
\end{tikzpicture}
\caption{The decomposition of components of the graph $H$.}\label{fig:CC}
\end{figure}

The vertex set $V(C_i)$ which is connected in the graph $H$,
is decomposed into the subsets $V_{i1}$ and~$V_{i2}$ in the graph $H'$,
one of which might be empty.
Therefore, we infer that 
\begin{align}
\label[rl]{rl:M'}
E_{C_i}(V_{i1},\,V_{i2})&\;\subseteq\; E(H)-E(H')\;\subseteq\;M'.
%,\rmand\\
%\label[rl]{rl:M}
%E_{H'}(V_{1j},\,V_{2j})&\;\subseteq\; E(H')-E(H)\;\subseteq\;M.
\end{align}
Let $i,j\in[2]$. From \cref{rl:M'}, we deduce that in the component $C_i$,
every vertex (if it exists) in the set $V_{ij}$ has at most one neighbor in the set $V_{ij'}$, where $j'\neq j$.
Therefore, we have
\begin{equation}\label[ineq]{delta:Vij}
\delta_H(V_{ij})\gge\delta(H)-1\gge \Bcl{\nf}-1,\qquad\text{if\; $V_{ij}\ne\emptyset$.}
\end{equation}
It follows that 
\begin{equation}\label[ineq]{Vij>=n/4}
|V_{ij}|\gge \Bcl{\nf},\qquad\text{if\; $V_{ij}\ne\emptyset$.}
\end{equation}

\smallskip
By way of contradiction,
assume that $V_{11}\neq\emptyset$. First, we claim that 
\[
V(C_1')\=V(C_1)
\rmand
V(C_2')\=V(C_2),
\]
that is, $V_{12}=V_{21}=\emptyset$.
In fact, if $V_{12}\ne\emptyset$, then \cref{Vij>=n/4} implies that
\[
|C_1|\=|V_{11}|+|V_{12}|\gge \nf+\nf\=\nt.
\]
Since $|C_1|\le |C_2|=n-|C_1|\le n/2$, we infer that $|C_1|=n/2$, i.e., 
the equality in the above inequality holds. 
In particular, the odd component~$C_1$ 
is composed of two vertex sets~$V_{11}$ and~$V_{12}$ of the same order, which is absurd!
This proves $V_{12}=\emptyset$, i.e., $V(C_1)=V_{11}$.
Now, if $V_{21}\ne\emptyset$, then 
\cref{Vij>=n/4,lb:c1:lem} imply that
\[
|C_1'|\=|V_{11}\cup V_{21}|\=|C_1|+|V_{21}|\gge \Bg{\nf+1}+\nf\>\nt.
\]
It follows that $|C_2'|<|C_1'|$, contradicting the premise $|C_1'|\le |C_2'|$. 
This proves the claim.

From \cref{cond:lem}, there exists an edge 
\[
e'\;\in\; E_M(C_1,C_2)-M'\;\subseteq\; E(H').
\]
From the claim, we see that 
\[
e'\;\in\; E_M(C_1,C_2)\=E_M(C_1',C_2').
\]
Combining the above two relations, we obtain
\begin{equation}\label{lem:contradiction}
e'\;\in\; E(H')\cap E_M(C_1',C_2')\sse E_{H'}(C_1',C_2').
\end{equation}
This is impossible since the components $C_1'$ and $C_2'$ are disconnected in the graph~$H'$.
This proves $V_{11}=\emptyset$. 
\smallskip

It remains to show that $V_{22}\neq\emptyset$.
In fact, the opposite relation $V_{22}=\emptyset$ implies that 
\[
V(C_2')\=V(C_1)
\rmand
V(C_1')\=V(C_2),
\] 
which results the same contradiction \eqref{lem:contradiction}.
This proves \cref{lem:CC:s=0}.
\end{proof}

Here is our main result.

\begin{thm}\label{thm:main}
Let $n\ge 34$. Then every $\{D_n,\,D_n+1\}$-graph of order $n$ has at least $\cl{n/4}$ 
disjoint perfect matchings. 
\end{thm}

\begin{proof}
Let $n\ge 34$. For short, we denote $D=D_n$ throughout this proof.
Let $G$ be an $\{D,D+1\}$-graph with a maximum family $\M$ of perfect matchings.
Let $l=|\M|$. At the beginning, we suppose that $n\ge 2$.

By way of contradiction, we assume $l\le \cl{n/4}-1$. It follows that 
\begin{equation}\label[ineq]{lb:D-l}
D-l\gge \Bcl{\nf}.
\end{equation}
Since $n\ge 34$, by \cref{lb:D-l}, we have
\begin{equation}\label[ineq]{D-l>=9}
D-l\gge 9.
\end{equation}
Let $H=G-\M$ denote the graph obtained by removing all edges constituting the matchings in the family~$\M$. 
Then the graph~$H$ is $\{D-l,\,D-l+1\}$-regular. Thus for any vertex $v$, we have
\begin{equation}\label[ineq]{range:deg}
D-l\lle \deg_{H}(v)\lle D-l+1.
\end{equation}

By the choice of the family~$\M$, the graph~$H$ has no perfect matchings.
By \cref{thm:Berge}, there is a vertex subset~$S$ 
such that the graph $H-S$ consists of factor-critical components
$C_1,C_2,\ldots,C_q$ with
\begin{align}
\label[ineq]{q>=s+2}
&q\gge s+2,\\
\label{mod:qs}
&q\equiv s\pmod{2},\\
\label{mod:ci}
&c_i\equiv 1\pmod{2},\rmand\\
\label[ineq]{order:c}
&1\lle c_1\lle c_2\lle\cdots\lle c_q,
\end{align}
where $s=|S|$ and $c_i=|C_i|$. 
By using \cref{range:deg}, we infer that
\begin{equation}\label[ineq]{ub:PS}
\sum_{i=1}^q |\partial_{H} C_i|
\=|\partial_{H} S|
\lle (D-l+1)\cdot s.
\end{equation}
On the other hand, by counting the vertices in~$H$, we find 
\begin{equation}\label{count:vertex}
n\=s+\sum_{i=1}^q c_i,
\end{equation}
Together with \cref{order:c,q>=s+2}, we infer that $n\ge s+q\ge 2s+2$, that is,
\begin{equation}\label[ineq]{s<=n/2-1}
s\lle \frac{n}{2}-1.
\end{equation}

Let $i\in[q]$.
Since every vertex in the component $C_i$ has at most $(c_i-1)$ neighbors inside itself,
it has at least $(D-(c_i-1))$ neighbors outside. Thus we have
\begin{equation}\label[ineq]{lb:partialCi}
|\partial_G C_i|\gge c_i\cdot (D-c_i+1).
\end{equation}
Along the same line, we can deduce 
\[
|\partial_{H} C_i|\gge c_i\cdot (D-l+1-c_i).
\]
Regarding the right hand side of the above inequality as a quadratic function in the variable~$c_i$, 
we obtain 
\begin{align}
|\partial_{H} C_i|\gge D-l,&\qquad\text{if $1\le c_i\le D-l$;}\label[ineq]{ineq:pc1}\\[3pt]
|\partial_{H} C_i|\gge 2(D-l-1),&\qquad\text{if $3\le c_i\le D-l-1$;}\rmand\label[ineq]{ineq:pc2}\\[3pt]
|\partial_{H} C_i|\gge 3(D-l-2),&\qquad\text{if $3\le c_i\le D-l-2$.}\label[ineq]{ineq:pc3}
\end{align}
In this proof, we often make effort to find the range of some order $c_i$
so as to use the corresponding lower bound of the number $|\partial_{H} C_i|$ given by one of \cref{ineq:pc1,ineq:pc2,ineq:pc3}.

Assume that $c_q\le D-l$, then \cref{order:c} implies that $1\le c_i\le D-l$ for all $i\in[q]$. 
Thus, \cref{ub:PS,ineq:pc1,q>=s+2} imply that 
\[
(D-l)\cdot (s+2)
\lle (D-l)\cdot q
\lle \sum_{i=1}^q |\partial_{H} C_i|
\lle (D-l+1)\cdot s.
\]
Simplifying it, and by using \cref{lb:D-l}, we find
$s\ge 2(D-l)\ge n/2$, contradicting \cref{s<=n/2-1}. 
Therefore, we have $c_q\ge D-l+1$. By using \cref{lb:D-l} again, we can deduce
\begin{equation}\label[ineq]{lb:cq}
c_q\gge D-l+1\gge \nf+1.
\end{equation}
Together with \cref{count:vertex,q>=s+2}, we infer that
\[
n\=s+\sum_{i=1}^{q-1}c_i+c_q
\gge s+(q-1)+\Bg{\nf+1}
\gge 2s+\nf+2,
\]
that is,
\begin{equation}\label[ineq]{ub:s}
s\lle \frac{3n}{8}-1.
\end{equation}

\bigskip

%For convenience, we write $E(X,Y)=E_G(X,Y)$
%for any vertex sets or subgraphs $X$ and $Y$.
%Similarly, we omit the subscript $G$ in the notations $e_G(X,Y)$ and $\partial_G X$.

Below we will handle the cases $s=1$, $s\ge 2$, and $s=0$, individually.
As will be seen, the case $s=1$ is relatively easy, 
the case $s=2$ implies that $s\ge \cl{n/4}$, 
and the case $s=0$ is proved to be reducible to the previous cases.

\begin{case}\label{case:s>1} $s\geq 2$.

First, we show that $s\ge \cl{n/4}$ in this case, and figure out some basic relation among the parameters.

\begin{clm}\label{clm1:s>1}
Suppose that $s\gge 2$.
Then we have 
\begin{itemize}
\vskip2pt\item[(i)]
$s\ge D-l\ge \cl{n/4}$;
\vskip2pt\item[(ii)]
$q=s+2$;
\vskip2pt\item[(iii)]
$c_i=1$ for $i\in[q-1]$;
\vskip2pt\item[(iv)]
$c_q=n-2s-1\in
%[D-l+1,\, n-2D+2l-1]\subseteq
[\,n/4+1,\,n/2-1\,]$.
\vskip2pt\item[(v)]
$|\partial_{H} C_q| \le s+l-D$, and the subgraph $C_q$ is Hamiltonian-connected. 
%\le n/8-1$.
\end{itemize}
\end{clm}

We shall show the above results one by one.

\medskip
\noindent{\bf (i).}
In order to show the desired lower bound $D-l$ of the number $s$, 
we suppose, to the contrary, that $s<D-l$. 
If the component $C_1$ consists of a single vertex, 
then all neighbors of this vertex lie in the set $S$.
As a consequence, by \cref{range:deg},
the set $S$ contains at least $D-l$ vertices, a contradiction.
Note that all the components~$C_i$ are of odd order.
Therefore, we have 
\begin{equation}\label[ineq]{c1>=3}
c_1\gge 3.
\end{equation}
It will be used to judge the condition when we apply \cref{ineq:pc2,ineq:pc3}.

From \cref{q>=s+2}, we see that $q\ge 4$.
Thus the notation $C_{q-3}$ is well defined.
Assume that $C_{q-3}\ge D-l$. By \cref{count:vertex,order:c,lb:cq}, we have 
\[
n\gge c_{q-3}+c_{q-2}+c_{q-1}+c_q\gge 3(D-l)+(D-l+1),
\]
contradicting \cref{lb:D-l}. Thus, we have $C_{q-3}\le D-l-1$. Together with \cref{c1>=3}, we find
\begin{equation}\label[ineq]{range:c:q-3}
3\lle c_i\lle D-l-1,\qquad \text{for all } i\in[q-3].
\end{equation}
Therefore, by using \cref{ineq:pc2}, we can deduce from \cref{ub:PS} that 
\begin{equation}\label[ineq]{ineq:2}
(D-l+1)s
\gge \sum_{i=1}^{q} |\partial_{H} C_i|
\gge \sum_{i=1}^{q-3} |\partial_{H} C_i|
\gge 2(D-l-1)(q-3).
\end{equation}

Assume that $q\ge s+3$. Then \cref{ineq:2} implies $D-l+1\gge 2(D-l-1)$, contradicting \cref{D-l>=9}.
This proves that $q\le s+2$.
In view of \cref{q>=s+2}, we derive that $q=s+2$.
Consequently, \cref{ineq:2} implies that
\[
s\lle 2\biggl(1+{2\over D-l-3}\biggr)\lle \frac{8}{3}.
\]
Therefore, we find $s=2$ and $q=4$.
%By \cref{order:c,lb:D-l}, we deduce that 
%\begin{equation}\label[ineq]{c1<=}
%c_1\lle \Bfl{\frac{n-2}{4}} \= \Bcl{\nf}-1\lle D-l-1.
%\end{equation}

Assume that $c_1\le D-l-2$.
By using \cref{ineq:pc2,ineq:pc3}, we can deduce from \cref{ub:PS} that 
\[
2(D-l+1)
\gge |\partial_{H} C_1|
\gge 3(D-l-2),
\]
contradicting \cref{D-l>=9}. From \cref{range:c:q-3}, we deduce that 
\[
c_1\=D-l-1\gge \Bcl{\nf}-1.
\]
In view of \cref{count:vertex} that $n-2=\sum_{i=1}^4c_i$, we find
\[
c_1\=c_2\=c_3\=c_4\=\frac{n-2}{4},
\]
contradicting \cref{lb:cq}.
This completes the proof of the lower bound part $s\ge D-l$ in \cref{clm1:s>1}~(i).
By \cref{lb:D-l} again, we obtain $s\ge \cl{n/4}$ immediately.

\medskip
\noindent{\bf (ii).}
Note that \cref{count:vertex,q>=s+2,order:c} give that
\begin{equation}\label[ineq]{ineq:5}
n\=s+\sum_{i=1}^{q-2}c_i+(c_{q-1}+c_q)
\gge s+(q-2)+2c_{q-1}
\gge 2(s+c_{q-1}).
\end{equation}
Together with the inequality $s\ge D-l$ confirmed in \cref{clm1:s>1}~(i), and \cref{lb:D-l}, we find that 
\[
c_{q-1}\lle \nt-D+l\lle D-l.
\]
Therefore, \cref{ub:PS,ineq:pc1} give
\[
(D-l+1)s
\gge\sum_{i=1}^{q} |\partial_{H} C_i|
\gge\sum_{i=1}^{q-1} |\partial_{H} C_i|
\gge (D-l)(q-1),
\]
which can be recast as $(D-l)(q-s-1)\le s$.
By using \cref{s<=n/2-1}, we infer that 
\[
q-s-1\lle {s\over D-l}\lle {n/2-1\over n/4}\<2.
\]
It follows that $q\le s+2$.
In view of \cref{q>=s+2}, we derive that $q=s+2$.

\medskip
\noindent{\bf (iii).} 
Suppose to the contrary that $c_{q-1}\ge 3$.

If $c_{q-1}\le D-l-1$, then \cref{ub:PS,ineq:pc1,ineq:pc2} yield that
\[
(D-l+1)s\gge\sum_{i=1}^{q-2} |\partial_{H} C_i|+|\partial_{H} C_{q-1}|
\gge (D-l)s+2(D-l-1),
\] 
that is, $s\ge 2(D-l-1)\ge n/2-2$.
Therefore, \cref{ineq:5} implies
$n\ge 2(s+3)\ge n+2$,
a contradiction. Therefore, we have 
$c_{q-1}\ge D-l$.
Together with \cref{clm1:s>1}~(i) that $s\geq D-l$, we see that all the equalities in \cref{ineq:5} hold true.
In particular, one has $c_q=n/4$, contradicting \cref{lb:cq}.
This confirms \cref{clm1:s>1}~(iii).

\medskip
\noindent{\bf (iv).}
Now, by \cref{clm1:s>1}~(ii) and (iii), \cref{count:vertex} reduces to
\[
n\=s+(q-1)+c_q\=2s+1+c_q,
\]
which gives the desired formula for~$c_q$. 
By using \cref{lb:D-l} and using $s\ge D-l$ from \cref{clm1:s>1}~(i), 
we find the desired upper bound $n/2-1$ of~$c_q$. 
The lower bound has been shown in \cref{lb:cq}.
This proves \cref{clm1:s>1}~(iv).

\medskip
\noindent{\bf (v).}
From \cref{clm1:s>1}~(iii) and \cref{ineq:pc1}, we infer that $|\partial_{H}C_i|\ge D-l$ for all $i\in[q-1]$.
Together with \cref{clm1:s>1}~(i), (ii), and \cref{ub:PS}, we deduce that 
\[
|\partial_{H} C_q|
\lle (D-l+1)s-(q-1)(D-l)
\=s+l-D.
%\lle \frac{n}{8}-1.
\]
Together with \cref{range:deg} and \cref{clm1:s>1}~(i) and (iv), we infer that
\begin{align*}%\label[ineq]{Cq_mindeg} 
\delta_{C_q}(C_q)
&\gge \delta_H(C_q)-|\partial_HC_q |\\[5pt]
&\gge (D-l)-(s+l-D)
\=2D-2l-s\\
&\gge D-l
\>\frac{c_q}{2}.
\end{align*}
By \cref{cor:Ore}, the subgraph $C_q$ is Hamiltonian-connected. 

This completes the proof of \cref{clm1:s>1}.
\qed
\medskip

\begin{clm}\label{clm:M:s>1}
There exists a matching $M_0\in \M$ such that 
\begin{equation}\label[ineq]{def:M0:s>=2}
|\partial_{M_0}C_q|\gge 3.
\end{equation} 
%In particular, we have $l\ge 1$.
\end{clm}

By \cref{clm1:s>1}~(iv), we see that $n/4+1\le c_q\le D$.
Therefore, \cref{lb:partialCi} implies that
\[
|\partial_GC_q|\gge c_q(D-c_q+1)\gge D.
\]
Assume that $|\partial_MC_q|\leq 1$ for all $M\in \M$. 
By using \cref{clm1:s>1}~(v), we deduce that
\[
s-D+l
\gge|\partial_HC_q|
\=|\partial_GC_q|-\sum_{M\in\M}|\partial_MC_q|
\gge D-l,
\]
which implies that $s\ge n/2$ by \cref{lb:D-l}, contradicting \cref{s<=n/2-1}.
Hence, there exists a matching $M_0\in\M$ such that $|\partial_{M_0}C_q|\ge 2$.
Since the component $C_q$ is of odd order, the cardinality $|\partial_MC_q|$ is odd for all matchings $M$.
Thus $|\partial_{M_0}C_q|\ge 3$. This proves \cref{clm:M:s>1}.
\qed
\bigskip

Denote $U=\cup_{i=1}^{q-1}V(C_i)$. 
From \cref{clm1:s>1}~(iii), we see that the set~$U$ consists of $(s+1)$ isolated vertices in the graph~$H$.
Now the graph $H$ has three parts $S$, $U$, and $C_q$.
Denote by $F$ the bipartite graph with vertex parts $S$ and $U$,
and with edge set $E_{H}(S,U)$. It can be obtained alternatively from the graph $H-C_q$ by removing 
the edges among vertices in the set~$S$.
 
By \cref{clm:M:s>1}, we can take a matching $M_0\in\M$ subject to \cref{def:M0:s>=2}.
Since the perfect matching $M_0$ covers the vertices of the set $U$, we have
\begin{equation}\label{dcps:U}
s+1\=|U|\= e_{M_0}(U,S)+e_{M_0}(U,C_q)+2e_{M_0}(U,U).
\end{equation}
For the same reason, we have 
\begin{equation}\label[ineq]{dcps:S}
s\=e_{M_0}(S,U)+e_{M_0}(S,C_q)+2e_{M_0}(S,S)
\gge e_{M_0}(S,U)+e_{M_0}(S,C_q).
\end{equation}
Subtracting \cref{dcps:U} from \cref{dcps:S}, and by using \cref{def:M0:s>=2}, we obtain
\[
-1\gge e_{M_0}(S,C_q)-e_{M_0}(U,C_q)-2e_{M_0}(U,U)\gge 3-2e_{M_0}(U,C_q)-2e_{M_0}(U,U).
\]
It follows that 
\begin{equation}\label[ineq]{lb:UU}
e_{M_0}(U,U)\gge 2-e_{M_0}(U,C_q). 
\end{equation}

Below we have three subcases to treat.
In each of them, we will apply \cref{lem:Hall} twice, 
taking $k\in\{0,1\}$ and $d\in\{D-l,\,D-l-1\}$.
Here we verify the condition $d\ge(s+k)/2+1$ and $d>=k+1$, as 
\begin{equation}\label{check:cond:s>=2}
D-l-1\gge \frac{s+1}{2}+1
\rmand
D-l-1\gge 2,
\end{equation}
whose truth can be seen directly from \cref{lb:D-l,ub:s,D-l>=9}.
In this way, we obtain two disjoint perfect matchings in the graph $H\cup M_0$,
contradicting the choice the family~$\M$.

\begin{subcase}\label{subcase:e>1}  
Suppose that $e_{M_0}(U,C_q)\ge 2$.
\smallskip

Let $e_{21}, e_{22}\in E_{M_0}(U,C_q)$. 
Note that we use the first subscript $2$ to indicate we are in the subcase with the assumption $e_{M_0}(U,C_q)\ge 2$.
See \cref{fig:subcase1}.
\begin{figure}[htbp]
\centering
\begin{tikzpicture}
%\draw[help lines,use as bounding box] (0, 0) grid (12,5);
\coordinate (U) at (-.5, .5);
\coordinate (S) at (-.5, 3.5);
\coordinate (Cq) at (11.5, 2);
\coordinate (u1) at (1, .5);
\coordinate (u2) at (2, .5);
\coordinate (u3) at (3, .5);
\coordinate (u4) at (4, .5);
\coordinate (s2) at (2, 3.5);
\coordinate (s3) at (3, 3.5);
\coordinate (s4) at (4, 3.5);
\coordinate (cu) at (8, .5);
\coordinate (cs) at (8, 3.5);
\coordinate (eSC) at (6, 3.5);
\coordinate (eUC) at (6, .5);
\coordinate (eUSC) at (6, .5);
\coordinate (M) at (2.75, 2);
\draw[fill=yellow!20] (9,2) ellipse [x radius=2cm, y radius=2.5cm]; %Cq
\draw[fill=red!20] (3,3.5) ellipse [x radius=2cm, y radius=1cm]; %S
\draw[fill=blue!20] (2.5,.5) ellipse [x radius=2.5cm, y radius=1cm]; %U
\draw[very thick]%matching
(u4) -- (cu)
;
\draw[very thick] decorate [decoration=snake, segment length=2mm] {(u3) -- (cs)};
\draw decorate [decoration=random steps]%Hamilton path
{(cs) -- ++(2,0) -- ++(0,-1) -- ++(-2,0) -- ++(0,-1) -- ++(2,0) -- ++(0,-1) -- (cu)
(2.5, 1.5) -- (3, 2.5)};
\draw
(S)node[]{$S$}
(U)node[]{$U$}
(Cq)node[]{$C_q$}
(M)node[right]{$M_{2i}'$}
(9.5,2)node[]{$P_2=M_{21}\cup M_{22}$}
(eUC)node[above]{$e_{21}$}
;
\node [label=above:$e_{22}$] at ($(u3)! .5 !(cs)$) {};
\foreach \point in 
{u3,u4,
cu,cs}
\fill[black,opacity=1] (\point) circle (3pt);
\end{tikzpicture}
\caption{The perfect matchings $M_{21}\cup M_{21}'\cup \{e_{21}\}$ and 
$M_{22}\cup M_{22}'\cup \{e_{22}\}$.}\label{fig:subcase1}
\end{figure}
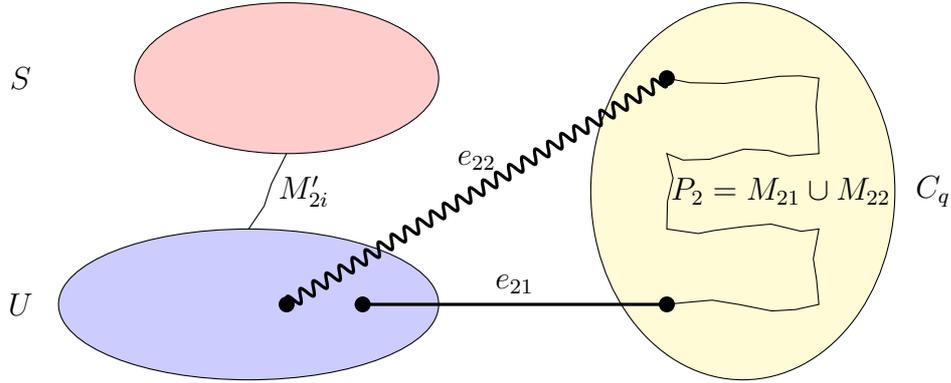

By \cref{clm1:s>1}~(v),
the component $C_q$ has a Hamiltonian path, say,~$P_2$, from the vertex $V(e_{21})\cap V(C_q)$ 
to the vertex $V(e_{22})\cap V(C_q)$.
For $i=1,2$, since the path $P_2-V(e_{2i})$ has an even number of vertices, 
it has a unique perfect matching, say,~$M_{2i}$.

In \cref{lem:Hall}, we take
\[
d=D-l,\quad
k=0,\quad
G'=F,\quad
S'=\emptyset,\rmand
U'=V(e_{21})\cap U.
\]
In the graph $F$, by \cref{range:deg},
every vertex in the set~$S$ has degree at most $(D-l+1)$, and the minimum degree $\delta_F(U)$ is at least $(D-l)$.
In view of \eqref{check:cond:s>=2}, 
we infer from \cref{lem:Hall} that the graph $F-V(e_{21})$ has a perfect matching, say,~$M_{21}'$.
Now, we take
\[
d=D-l-1,\quad
k=0,\quad
G'=F-M_{21}',\quad
S'=\emptyset,\rmand
U'=V(e_{22})\cap U.
\]
Consider the graph $F-M_{21}'$.
Since the matching $M_{21}'$ is perfect, by \cref{range:deg},
every vertex in the set~$S$ has degree at most $(D-l)$,
and that the minimum degree $\delta_{F-M_{21}'}(U)$ is at least $(D-l-1)$.
Again, \cref{lem:Hall} provides a perfect matching~$M_{22}'$ of the graph $F-V(e_{22})-M_{21}'$.

From definition, we obtain two disjoint perfect matchings
\[
M_{2i}''\=M_{2i}\cup M_{2i}'\cup \{e_{2i}\}\qquad (i=1,2),
\]
of the graph $H\cup M_0$. 
As a consequence, the family 
$(\M-M_0)\cup \{M_{21}'',\,M_{22}''\}$ consists of $(l+1)$ disjoint perfect matchings,
contradicting the choice of the family~$\M$.
\end{subcase}

\begin{subcase}\label{subcase:e=0} 
Suppose that $e_{M_0}(U,C_q)=0$.
\smallskip

In this case, by \cref{def:M0:s>=2}, we have $e_{M_0}(S,C_q)\ge 3$.
Thus we can choose two edges $e_{01},e_{02}\in E_{M_0}(S,C_q)$. 
See \cref{fig:subcase2}.
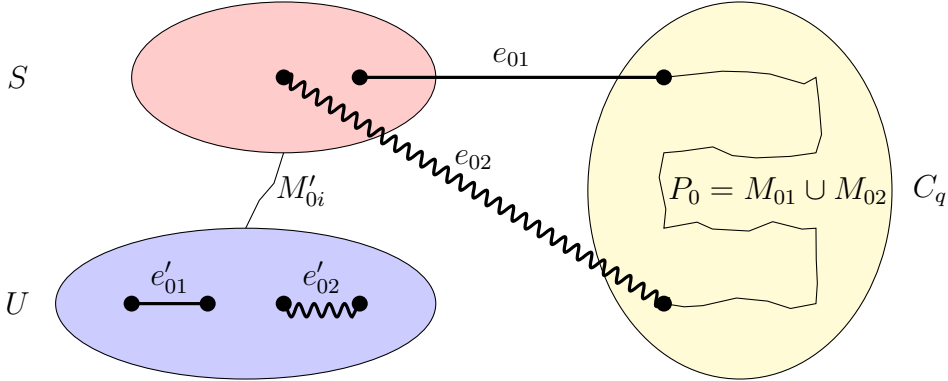
\begin{figure}[htbp]
\centering
\begin{tikzpicture}
%\draw[help lines,use as bounding box] (0, 0) grid (12,5);
\coordinate (U) at (-.5, .5);
\coordinate (S) at (-.5, 3.5);
\coordinate (Cq) at (11.5, 2);
\coordinate (u1) at (1, .5);
\coordinate (u2) at (2, .5);
\coordinate (u3) at (3, .5);
\coordinate (u4) at (4, .5);
\coordinate (s2) at (2, 3.5);
\coordinate (s3) at (3, 3.5);
\coordinate (s4) at (4, 3.5);
\coordinate (cu) at (8, .5);
\coordinate (cs) at (8, 3.5);
\coordinate (eSC) at (6, 3.5);
\coordinate (eUC) at (6, .5);
\coordinate (eUSC) at (6, .5);
\coordinate (e01') at (1.5, .5);
\coordinate (e02') at (3.5, .5);
\coordinate (M) at (2.75, 2);
\draw[fill=yellow!20] (9,2) ellipse [x radius=2cm, y radius=2.5cm]; %Cq
\draw[fill=red!20] (3,3.5) ellipse [x radius=2cm, y radius=1cm]; %S
\draw[fill=blue!20] (2.5,.5) ellipse [x radius=2.5cm, y radius=1cm]; %U
\draw[very thick]%matching
(s4) -- (cs)
(u1) -- (u2)
;
\draw[very thick] decorate [decoration=snake, segment length=2mm] {(s3) -- (cu) (u3) -- (u4)};
\draw decorate [decoration=random steps]%Hamilton path
{(cs) -- ++(2,0) -- ++(0,-1) -- ++(-2,0) -- ++(0,-1) -- ++(2,0) -- ++(0,-1) -- (cu)
(2.5, 1.5) -- (3, 2.5)};
\draw
(S)node[]{$S$}
(U)node[]{$U$}
(Cq)node[]{$C_q$}
(M)node[right]{$M_{0i}'$}
(9.5,2)node[]{$P_0=M_{01}\cup M_{02}$}
(eSC)node[above]{$e_{01}$}
(e01')node[above]{$e_{01}'$}
(e02')node[above]{$e_{02}'$};
\node [label=above:$e_{02}$] at ($(s3)! .5 !(cu)$) {};
\foreach \point in 
{u1,u2,u3,u4,
s3,s4,
cu,cs}
\fill[black,opacity=1] (\point) circle (3pt);
\end{tikzpicture}
\caption{The perfect matchings $M_{0i}\cup M_{0i}'\cup \{e_{0i},\,e_{0i}'\}$ ($i=1,2$).}\label{fig:subcase2}
\end{figure}

By \cref{clm1:s>1}~(v),
the component $C_q$ has a Hamiltonian path, say,~$P_0$, from the vertex $V(e_{01})\cap V(C_q)$ 
to the vertex $V(e_{02})\cap V(C_q)$.
Same to \cref{subcase:e>1}, for $i=1,2$,
we denote by $M_{0i}$ the unique perfect matching of the path $P_0-V(e_{0i})$.
From \cref{lb:UU}, we infer that $e_{M_0}(U,U)\gge 2$.
Thus, we can pick edges $e_{01}',e_{02}'\in E_{M_0}(U,U)$.
In \cref{lem:Hall}, we take
\[
d=D-l,\quad
k=1,\quad
G'=F,\quad
S'=V(e_{01})\cap S,\rmand
U'=V(e_{01}').
\]
Same to \cref{subcase:e>1}, the graph $F-V(e_{01})-V(e_{01}')$ has a perfect matching, say,~$M_{01}'$.
Then, we take
\[
d=D-l-1,\quad
k=1,\quad
G'=F-M_{01}',\quad
S'=V(e_{02})\cap S,\rmand
U'=V(e_{02}').
\]
Note that in the graph $F-M_{01}'$, 
the vertex in the set $V(e_{01})\cap S$ has degree at most $(D-l+1)$,
every other vertex in the set~$S$ has degree at most $(D-l)$, 
and that the minimum degree $\delta_{F-M_{01}'}(U)$ is at least $(D-l-1)$.
Again, \cref{lem:Hall} offers a perfect matching~$M_{02}'$ of the graph $F-V(e_{02})-V(e_{02}')$.
From definition, we obtain two disjoint perfect matchings
$M_{0i}\cup M_{0i}'\cup \{e_{0i},\,e_{0i}'\}$ $(i=1,2)$
of the graph $H\cup M_0$, the same contradiction as in Subcase 1.1.
\end{subcase}

\begin{subcase}\label{subcase:e=1}
Suppose that $e_{M_0}(U,C_q) = 1$.
\smallskip

In this case, we can choose an edge $e_{11}\in E_{M_0}(U,C_q)$. See \cref{fig:subcase3}.
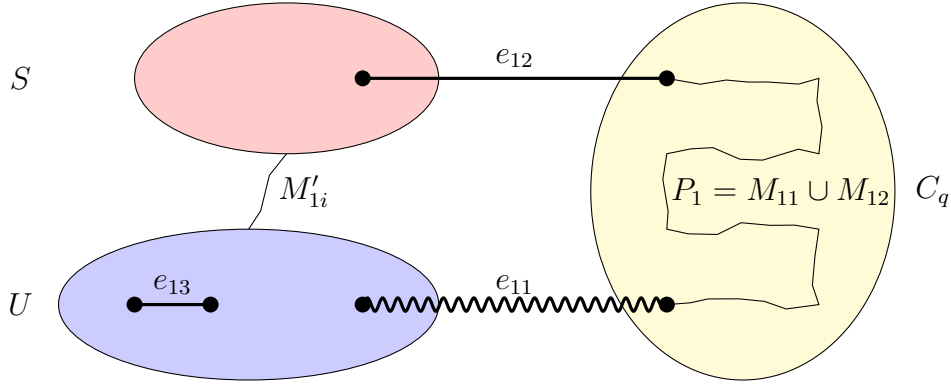
\begin{figure}[htbp]
\centering
\begin{tikzpicture}
%\draw[help lines,use as bounding box] (0, 0) grid (12,5);
\coordinate (U) at (-.5, .5);
\coordinate (S) at (-.5, 3.5);
\coordinate (Cq) at (11.5, 2);
\coordinate (u1) at (1, .5);
\coordinate (u2) at (2, .5);
\coordinate (u3) at (3, .5);
\coordinate (u4) at (4, .5);
\coordinate (s2) at (2, 3.5);
\coordinate (s3) at (3, 3.5);
\coordinate (s4) at (4, 3.5);
\coordinate (cu) at (8, .5);
\coordinate (cs) at (8, 3.5);
\coordinate (eSC) at (6, 3.5);
\coordinate (eUC) at (6, .5);
\coordinate (eUSC) at (6, .5);
\coordinate (e01') at (1.5, .5);
\coordinate (e02') at (3.5, .5);
\coordinate (M) at (2.75, 2);
\draw[fill=yellow!20] (9,2) ellipse [x radius=2cm, y radius=2.5cm]; %Cq
\draw[fill=red!20] (3,3.5) ellipse [x radius=2cm, y radius=1cm]; %S
\draw[fill=blue!20] (2.5,.5) ellipse [x radius=2.5cm, y radius=1cm]; %U
\draw[very thick]%matching
(s4) -- (cs)
(u1) -- (u2)
;
\draw[very thick] decorate [decoration=snake, segment length=2mm] {(u4) -- (cu)};
\draw decorate [decoration=random steps]%Hamilton path
{(cs) -- ++(2,0) -- ++(0,-1) -- ++(-2,0) -- ++(0,-1) -- ++(2,0) -- ++(0,-1) -- (cu)
(2.5, 1.5) -- (3, 2.5)};
\draw
(S)node[]{$S$}
(U)node[]{$U$}
(Cq)node[]{$C_q$}
(M)node[right]{$M_{1i}'$}
(9.5,2)node[]{$P_1=M_{11}\cup M_{12}$}
(eUC)node[above]{$e_{11}$}
(eSC)node[above]{$e_{12}$}
(e01')node[above]{$e_{13}$};
\foreach \point in 
{u1,u2,u4,
s4,
cu,cs}
\fill[black,opacity=1] (\point) circle (3pt);
\end{tikzpicture}
\caption{The perfect matchings $M_{11}\cup M_{11}'\cup \{e_{11}\}$
and $M_{12}\cup M_{12}'\cup \{e_{12},\,e_{13}\}$.}\label{fig:subcase3}
\end{figure}

From \cref{def:M0:s>=2}, we infer that $e_{M_0}(C_q,S)\ge 2$,
which allows us to pick an edge $e_{12}\in E_{M_0}(C_q,S)$ such that $V(e{11})\cap V(e_{12})=\emptyset$.
Same to \cref{subcase:e>1}, let $P_1$ be a Hamiltonian path 
from the vertex $V(e_{11})\cap V(C_q)$ to the vertex $V(e_{12})\cap V(C_q)$.
Denote by~$M_{1i}$ the perfect matching of the path $P_1-V(e_{1i})$ for $i=1,2$.
Taking
\[
d=D-l,\quad
k=0,\quad
G'=F,\quad
S'=\emptyset,\rmand
U'=V(e_{11})\cap U,
\]
we infer from \cref{lem:Hall} that the graph $F-V(e_{11})$ has a perfect matching, say,~$M_{11}'$.
By \cref{lb:UU}, we have $e_{M_0}(U,U)\ge 1$. Let $e_{13}\in E_{M_0}(U,U)$.
Then, we put
\[
d=D-l-1,\quad
k=1,\quad
G'=F-M_{11}',\quad
S'=V(e_{12})\cap S,\rmand
U'=V(e_{13}).
\]
Again, \cref{lem:Hall} results in a perfect matching $M_{12}'$ of the graph $F-V(e_{11})-V(e_{12})-V(e_{13})$.
From definition, we obtain two disjoint perfect matchings
\[
M_{11}\cup M_{11}'\cup \{e_{11}\}
\rmand 
M_{12}\cup M_{12}'\cup \{e_{12},\,e_{13}\}
\] 
are disjoint perfect matchings of the graph $H\cup M_0$, the same contradiction.
\end{subcase}

This completes the proof for \cref{case:s>1}.
\end{case}

\bigskip

\begin{case}\label{case:s=1} $s=1$.
\medskip

Before dealing with the other cases $s=1$ and $s=0$,
we give some common properties for these two cases.
Let $j\in[q]$.
Every vertex in the subgraph~$H[C_j]$ has at most~$s$ neighbors outside~$C_j$.
Therefore,
by \cref{range:deg}, every vertex in~$H[C_j]$ has at least $(D-l-s)$ neighbors inside~$C_j$. In other words,
\begin{equation}\label[ineq]{delta:s<=1}
\delta_{C_j}(C_j)\gge D-l-s\gge \Bcl{\nf}-s.
\end{equation}
It follows that 
\begin{equation}\label[ineq]{lb:c:s=01}
c_j\gge \delta_{C_j}(C_j)+1\gge D-l-s+1\gge \Bcl{\nf}-s+1.
\end{equation}
From \cref{count:vertex} and that $s\in\{0,1\}$, we have
\[
n\=s+\sum_{j=1}^qc_j\gge s+q\cdot\Bg{\nf-s+1}
\>q\cdot\nf.
\]
It follows that $q\le 3$. From \cref{q>=s+2,mod:qs}, we infer that 
\begin{equation}\label{q=s+2:s=01}
q\=s+2.
\end{equation}

From \cref{clm:M:s>1}, we see that the graph $G$ has a perfect matching if $s\ge 2$.
In fact, this is also true for $s\in\{0,1\}$.

\begin{clm}\label{clm:l>0}
Let $s\in\{0,1\}$. Then the graph $G$ has a perfect matching, i.e., we have $l\ge 1$. 
\end{clm}

By \cref{count:vertex,order:c,delta:s<=1,q=s+2:s=01}, we find
\begin{equation}\label[ineq]{pf:l>0}
n\=s+\sum_{i=1}^qc_i\gge s+(s+2)\cdot c_1
\gge s+(s+2)\cdot (D-l-s+1).
\end{equation}
Assume that $l=0$.
For $s=1$, \cref{pf:l>0} implies $n\ge 1+3D\ge 1+3(n/2-1)$, contradicting $n\ge34$.
For $s=0$, \cref{pf:l>0} implies $n\ge 2(D+1)=4\cl{n/4}\ge n$. Thus the equality in \cref{pf:l>0} holds.
In particular, we have $n\equiv0\pmod{4}$ and $c_1=D+1=n/2$ is even, contradicting \cref{mod:ci}.
This proves \cref{clm:l>0}.
\qed

\medskip

From \cref{q=s+2:s=01}, we have $q=3$.
We rename the components $C_1$, $C_2$, and $C_3$ by $T_1$, $T_2$, and~$T_3$, so that 
\begin{equation}\label[ineq]{order:T}
e_{H}(S,\,T_3)\=\max_{1\le i\le 3} e_{H}(S,\,C_i).
%e_{H}(S,T_1)\le e_{H}(S,T_2)\le e_{H}(S,T_3).
\end{equation}
Denote $|T_i|=t_i$. 
This case $s=1$ will be handled by presenting a family of disjoint perfect matchings larger than $\M$.
To do this, we will discover a matching $M\in\M$ such that the graph $H\cup M$ has two disjoint perfect matchings.
\Cref{clm:HC:s=1,clm:M:s=1} will be of use.

\begin{clm}\label{clm:HC:s=1}
We have 
\begin{align*}
\Bcl{\nf}+1&\lle t_i\lle \nt-3,\qquad\text{for $i=1,2$}, \rmand\\[5pt]
\Bcl{\nf}&\lle t_3\lle \nt-3.
\end{align*}
As a consequence, every component~$T_j$ ($j=1,2,3$) is Hamiltonian-connected.
\end{clm}

From \cref{lb:c:s=01}, we obtain the desired lower bound of~$t_3$ directly.
Assume that $t_i=\cl{n/4}$ for some $i\in\{1,2\}$.
Let $S=\{v^*\}$.
By \cref{range:deg}, every vertex in the component~$T_i$ is a neighbor of the vertex~$v^*$.
Thus $e_{H}(S,T_i)\ge t_i$. Therefore, by \cref{order:T}, we have
\[
\deg_H(v^*)\=\sum_{j=1}^3e_H(S,\,T_j)\gge e_H(S,T_i)+e_H(S,T_3)\gge 2t_i\=2\Bcl{\nf}.
\]
By \cref{range:deg}, we find $l=0$, contradicting \cref{clm:l>0}.
Hence, both integers~$t_1$ and~$t_2$
have the lower bound $\cl{n/4}+1$.

By the lower bounds of $t_i$ that just obtained, we infer that
\[
t_3\=|G-S-T_1-T_2|\lle n-1-\Bg{\nf+1}-\Bg{\nf+1}\= \nt-3,
\]
the desired upper bound of $t_3$.
Along the same line, we have 
\[
t_1\=|G-S-T_2-T_3|\lle n-1-\Bg{\nf+1}-\nf\=\nt-2.
\]
If $t_1=n/2-2$, i.e., if the equality in the above inequality holds, 
then $t_2=n/4+1$ and $t_3=n/4$, having different parities.
But this is impossible since the order of every component $T_i$ has odd parity.
This confirms the desired upper bound of $t_1$.
The desired upper bound of $t_2$ can be shown in the same fashion. 

Let $j\in[3]$. By \cref{delta:s<=1}, we have
\[
2\delta_{T_j}(T_j)
\gge 2\Bg{\nf-1}
\gge t_j+1.
\]
By \cref{cor:Ore}, every component $T_j$ is Hamiltonian-connected.
This proves \cref{clm:HC:s=1}.
\qed

\begin{clm}\label{clm:M:s=1}
There is a matching $M\in\M$ such that $e_M(T_1,T_2)\ge2$.
\end{clm}

We estimate the number of edges between the sets $T_1\cup T_2$ and $S\cup T_3$.
On the one side, from \cref{range:deg,order:T}, we infer that 
\[
|\partial_{H}(S\cup T_3)|
\=\sum_{i=1}^{2}e_{H}(S,T_i)
\lle \frac{2}{3}\sum_{i=1}^{3}e_{H}(S,T_i)\=\frac{2}{3}\deg_{H}(v^*)
\lle \frac{2}{3}(D+1-l).
\]
Therefore, we have
\begin{align}
|\partial_G(S\cup T_3)|
&\=|\partial_{H}(S\cup T_3)|+|\partial_{G-H}(S\cup T_3)|
\lle|\partial_{H}(S\cup T_3)|+|S\cup T_3|\cdot|\M|\notag\\
&\lle \frac{2}{3}(D+1-l)+(n-t_1-t_2)\cdot l.\label[ineq]{pf:partial2}
\end{align}
On the other hand, assume that \cref{clm:M:s=1} is false. Then $e_M(T_1,T_2)\le1$ for every matching $M\in\M$.
It follows that 
\[
e_G(T_1,T_2)\=e_H(T_1,T_2)+e_{G-H}(T_1,T_2)
\=0+\sum_{M\in\M}e_{M}(T_1,T_2)
\lle |\M|\= l.
\] 
Therefore, we have
\begin{align}
|\partial_G(T_1\cup T_2)|
&\=\sum_{\mathrlap{v\in T_1\cup T_2}}\deg_G(v)-\sum_{i=1}^{2}\sum_{v\in T_i}\deg_{T_i}(v)-2e_G(T_1,T_2)\notag\\
&\gge D\cdot (t_1+t_2)-\sum_{i=1}^2t_i(t_i-1)-2l.\label[ineq]{pf:partial1}
\end{align}

Combining \cref{pf:partial1,pf:partial2} with the identity $\partial (T_1\cup T_2)=\partial (S\cup T_3)$, we infer that
\begin{equation}\label[ineq]{pf180}
\frac{2}{3}(D+1-l)+(n-t_1-t_2)\cdot l
-\biggl(D\cdot(t_1+t_2)-\sum_{i=1}^2t_i(t_i-1)-2l\biggr)\gge0.
\end{equation}
Since the coefficient of $l$ in the left hand side of \cref{pf180} is 
$-2/3+(n-t_1-t_2)+2>0$,
and since the coefficient of $D$ in the left hand side of the above inequality is 
$2/3-(t_1+t_2)<0$,
we can substitute $l$ by its upper bound $(n-2)/4$,
and substitute $D$ by its lower bound $n/2-1$ into \cref{pf180}, which gives
\begin{equation}\label[ineq]{pf152}
f(t_1)+f(t_2)+\Bg{\frac{n^2}{4}+\frac{n}{6}-\frac{2}{3}}\gge0,
\end{equation}
where 
\[
f(t)\=t^2+\Bigl(-\frac{3n}{4}+\frac{1}{2}\Bigr)t.
\]
From the domain of $t_i$ ($i=1,2$) obtained in \cref{clm:HC:s=1}, 
and since $n\ge 34$,
it is elementary to derive that the quadratic function $f(t_i)$ has upper bound $f(n/4+1)$.
%\[
%f(t_i)\lle f\Bigl(\frac{n}{4}+1\Bigr)\=-\frac{n^2}{8}-\frac{n}{2}.
%\]
From \cref{pf152}, we obtain
\[
2f\Bigl(\frac{n}{4}+1\Bigr)+\Bg{\frac{n^2}{4}+\frac{n}{6}-\frac{2}{3}}\gge0,
\]
which reduces to $n\le 28$, a contradiction to the premise $n\ge 34$. This proves \cref{clm:M:s=1}. \qed
\medskip

By \cref{clm:M:s=1}, we can suppose that $e_1,e_2\in E_M(T_1,T_2)$.
%$e_i=v_{1i}v_{2i}\in M$ for $i\in\{1,2\}$, where $v_{ki}\in V(T_k)$ for $k=1,2$.
By \cref{clm:HC:s=1}, the component~$T_i$ has a Hamiltonian path~$p_i$  
from the vertex $V(T_i)\cap V(e_1)$ to the vertex~$V(T_i)\cap V(e_2)$.
%~$v_{1i}$ to the vertex~$v_{2i}$.
Thus we obtain a Hamiltonian cycle $h_1=(p_1,\,e_2,\,p_2,\,e_1)$ of the subgraph $T_1\cup T_2\cup\{e_1,e_2\}$.
Since both the orders~$t_1$ and $t_2$ are odd, the length $(t_1+t_2)$ of the cycle $h_1$ is even.
See \cref{fig:s=1}.

On the other hand, from \cref{order:T,range:deg}, we have
\[
e_H(S,T_3)\gge \frac{1}{3}\deg_{H}(v^*)
\gge \frac{1}{3}\Bcl{\nf}.
\] 
Since $n\ge 34$, we have $e_H(S,T_3)\ge 3$.
Let~$v_{31}$ and~$v_{32}$ be two neighbors of the vertex~$v^*$ in the component~$T_3$.
By \cref{clm:HC:s=1} again, the component~$T_3$ has a Hamiltonian path~$p_3$ from
the vertex $v_{31}$ to the vertex $v_{32}$.
This gives a Hamiltonian cycle $h_2=(p_3,\,v_{32}v^*v_{31})$ of the subgraph $H[S\cup V(T_3)]$.
Since the order $t_3$ is odd, the length $t_3+1$ of the cycle~$h_2$ is even.

\begin{figure}[htbp]
\centering
\begin{tikzpicture}
\coordinate (e11) at (1.5, 1);
\coordinate (e12) at (3.5, 1);
\coordinate (e21) at (1.5, 3);
\coordinate (e22) at (3.5, 3);
\coordinate (e31) at (7.5, 1);
\coordinate (e32) at (7.5, 3);
\coordinate (v) at (9, 2);
\draw[fill=blue!20] (1, 2) ellipse [x radius=1cm, y radius=2cm]; %Cq
\draw[fill=red!20] (4, 2) ellipse [x radius=1cm, y radius=2cm]; %S
\draw[fill=yellow!20] (7, 2) ellipse [x radius=1cm, y radius=2cm]; %U
\draw[help lines,use as bounding box] (0, 0) grid (9, 4);
\draw[very thick]%matching
(e31) -- (v)
(e11) -- (e12)
;
\draw[very thick] decorate [decoration=snake, segment length=2mm] {(e32) -- (v) (e21) -- (e22)};
\draw decorate [decoration=random steps]%Hamilton path
{(e11) -- ++(-1, 0) -- (1, 2) -- (.5, 3) -- ++(1, 0)
(e12) -- ++(1, 0) -- (4, 2) -- (4.5, 3) -- ++(-1, 0)
(e32) -- ++(-1, 0) -- (7, 2) -- (6.5, 1) -- ++(1, 0)
};
\draw
(v)node[right=2mm]{$v^*$}
(e31)node[below]{$v_{31}$}
(e32)node[above]{$v_{32}$}
(2.5, 1)node[above]{$e_1$}
(2.5, 3)node[above]{$e_2$}
;
\foreach \point in 
{e11,e12,e21,e22,e31,e32,v}
\fill[black,opacity=1] (\point) circle (3pt);
\end{tikzpicture}
\caption{The perfect matching union $M_1\cup M_2$.}\label{fig:s=1}
\end{figure}
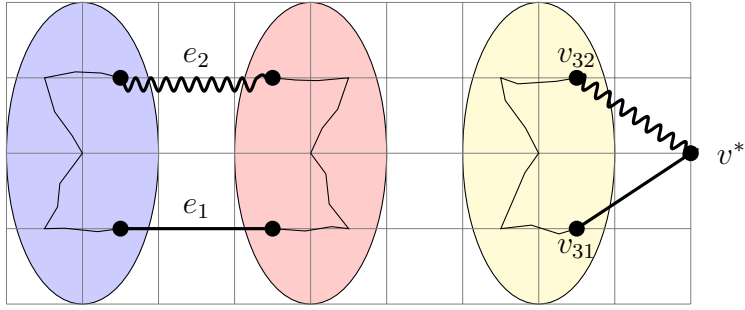

Note that the union of the even cycles $h_1$ and $h_2$ can be decomposed into two disjoint perfect matchings,
say, $M_1$ and $M_2$, of the graph $H\cup M$.
Then the family 
$(\M\cup \{M_1, M_2\})-M$
consists of $(l+1)$ disjoint perfect matchings, contradicting the choice of~$\M$.
This completes the proof for \cref{case:s=1}.
\end{case}

\begin{case}\label{case:s=0} $s=0$.
\medskip

From \cref{q=s+2:s=01}, we infer that $q=2$.
In other words, the graph $H$ consists of factor-critical components $C_1$ and $C_2$.
\Cref{clm:CC:s=0} will be used several times for solving \cref{case:s=0}.

\begin{clm}\label{clm:CC:s=0}
For any matching $M\in\M$ and for any perfect matching $M'$ of the graph $H\cup M$,
the graph $(H\cup M)-M'$ consists of two factor-critical components of orders at least $\cl{n/4}+1$.
\end{clm}

Let $M\in\M$, and let $M'$ be a perfect matching of the graph $H\cup M$.
From the choice of the family~$\M$, we infer that
the subgraph $(H\cup M)-M'$ has no perfect matchings.
By \cref{thm:Berge}, there is a vertex set~$S'$ 
such that the graph $H'-S'$ consists of $q'$ factor-critical components.
If $S'\ne\emptyset$, then one may consider the family $(\M-M)\cup\{M'\}$ of disjoint perfect matchings 
instead of the family~$\M$, as in the previous proofs for \cref{case:s>1,case:s=1}.
Therefore, we can suppose that $S'=\emptyset$.
Along the same lines, we are led to $q'=2$.
In analog with \cref{lb:c:s=01}, we find each component has order at least $\cl{n/4}+1$.
This proves \cref{clm:CC:s=0}.
\medskip

%Therefore, with \cref{lb:c1':s=0,Vij>=n/4}, we derive that
%\begin{equation}\label[ineq]{lb:c2:s=0}
%c_2\=|C_1'|+|V_{22}|\gge \Bg{\nf+1}+\nf\=\nt+1.
%\end{equation}
%It follows from \cref{count:vertex} that 
%\begin{equation}\label[ineq]{ub:c1:s=0:enhanced}
%c_1\lle \nt-1.
%\end{equation}
%This enhances \cref{ub:c1:s=0}. 

From \cref{lb:partialCi}, we infer that 
\begin{equation}\label[ineq]{sumM>=}
\sum_{M\in\M}e_M(C_1,C_2)
\=e_G(C_1,C_2)
\gge c_i(D-c_i+1)
\= c_i\cdot \bgg{2\Bcl{\nf}-c_i}.
\end{equation}
Since $c_1\le c_2$, we have $c_1\le n/2$.
If $c_1=n/2$, then the integer $n/2$, as the order of the factor-critical component, is odd.
Then \cref{sumM>=} becomes
\[
\sum_{M\in\M}e_M(C_1,C_2)
\gge c_i\cdot \Bg{\nt+1-c_i}
\=\nt.
\]
Otherwise, by \cref{lb:c:s=01}, we have $n/4+1\le c_1\le n/2-1$. In this case, \cref{sumM>=} implies
\[
\sum_{M\in\M}e_M(C_1,C_2)\gge c_i\cdot \Bg{\nt-c_i}\gge \nt-1.
\]
Anyway, the sum on the left hand side of \cref{sumM>=} is at least $n/2-1$.
Consequently, by \cref{clm:l>0} that $l\ge 1$, and by the assumption $l\le \cl{n/4}-1$, 
there exists a matching $M_0\in \M$ such that 
\[
e_{M_0}(C_1,\,C_2)\gge \frac{n/2-1}{l}\gge 2.
\]
Since the order $c_1$ is odd, and the matching $M_0$ is perfect, the integer $e_{M_0}(C_1,C_2)$ must be odd. 
Thus, the above lower bound can be enhanced to 
\begin{equation}\label[ineq]{def:M:s=0}
e_{M_0}(C_1,\,C_2)\gge 3.
\end{equation}
Let $e_0\in e_{M_0}(C_1,C_2)$. 
Since each of the components $C_i$ is factor-critical,
the subgraph $C_i-V(e_0)$ has a perfect matching, say, $M_{0i}$. 
Thus, the graph $H\cup M_0$ has the perfect matching 
\[
M_0'\=M_{01}\cup M_{02}\cup \{e_0\}.
\]
We further denote 
\begin{align*}
H'&\=(H\cup M_0)-M_0',\rmand\\
F&\=H'\cup M_0'\=H\cup M_0.
\end{align*}
By \cref{clm:CC:s=0}, we can suppose that the graph $H'$ consists of factor-critical components~$C_1'$ and~$C_2'$,
such that 
\begin{equation}\label[ineq]{lb:c1':s=0}
\nf+1\lle |C_1'|\lle |C_2'|.
\end{equation}
Denote
\[
V_{ij}\=V(C_i)\cap V(C_j').
\]
From \cref{def:M:s=0} and the definition of the matching $M_0'$, one may verify \cref{cond:lem} directly. 
Thus, by \cref{lem:CC:s=0}, we infer that 
\begin{equation}\label[rl]{C1'ssC2:s=0}
V(C_1')\ss V(C_2).
\end{equation}
On the other hand,
from \cref{lb:c:s=01,lb:c1':s=0}, we infer that 
\begin{equation}\label[ineq]{ub:V22}
|V_{22}|\=n-c_1-|C_1'|\lle n-\Bg{\nf+1}-\Bg{\nf+1}\=\nt-2.
\end{equation}
%Since $\partial_F V_{22}\subseteq M\cup M$, we infer that
From \cref{delta:Vij,ub:V22}, we infer that
\begin{equation}\label[ineq]{Hamilton:V22}
\delta_H(V_{22})\gge\nf-1\gge\frac{|V_{22}|}{2}.
\end{equation}
From \cref{C1'ssC2:s=0}, we see that $V_{22}\ne\emptyset$. 
%Since the orders $c_2$ and $|C_1'|$ are odd, the order $|V_{22}|=c_2-|C_1'|$ is even.
By \cref{Vij>=n/4} and the premise $n\ge 34$, we find $|V_{22}|\ge 9$.
By Dirac's \cref{thm:Dirac},
we conclude that the subgraph $H[V_{22}]$ is Hamiltonian.
Let~$H_{22}$ be a Hamiltonian cycle of the subgraph~$H[V_{22}]$.

We will find another perfect matching in the graph~$F$ in \cref{clm:M'':s=0}, based on \cref{clm:e1e1':s=0}.

\begin{clm}\label{clm:e1e1':s=0}
The graph $F$ contains two edges 
\[
e_1\in E_{M_0-e_0}(C_1,\,V_{22})
\rmand
e_1'\in E_{H}(C_1',\,V_{22}),
\]
%\begin{align}
%e_1&\in E_{M-e_0}(C_1,\,V_{22}),\rmand\label[rl]{def:e1:s=0}\\
%e_1'&\in E_H(C_1',\,V_{22}),\label[rl]{def:e1':s=0}
%\end{align}
such that $V(e_1)\cap V(e_1')=\emptyset$. 
\end{clm}

Recall that every factor-critical graph is $2$-edge-connected.
Since the component~$C_2$ is factor-critical, we infer that
\begin{equation}\label[ineq]{eC1'V22}
e_{H}(C_1',\,V_{22})\gge2.
\end{equation}
To show \cref{clm:e1e1':s=0}, it suffices to show that 
\begin{equation}\label[ineq]{dsr:e1:s=0}
e_{M_0-e_0}(C_1,\,V_{22})\gge 2.
\end{equation}
From the definition $M_0'=M_{01}\cup M_{02}\cup\{e_0\}$, we see that 
\[
E_{M_0}(C_1,C_2)\cap M_0'\=\{e_0\}.
\]
From the definition $H'=(H\cup M_0)-M_0'$, we can deduce that
\[
E_{M_0}(C_1,C_2)-e_0
\;\subset\; E(H').
%\=E_{H'}(C_1')\cup E_{H'}(C_2').
\]
By \cref{C1'ssC2:s=0}, 
we can enhanced the above relation to 
\[
E_{M_0}(C_1,C_2)-e_0\ss E(C_2').
\]
Consequently, we have
\[
E_{M_0}(C_1,C_2)-e_0\;\subset\; E(C_2')\cap E_{M_0-e_0}(C_1,C_2)\=E_{M_0-e_0}(C_1,\,V_{22}).
\]
Hence, the desired \cref{dsr:e1:s=0} follows from \cref{def:M:s=0}.
This proves \cref{clm:e1e1':s=0}.
\medskip

Let $e_1$ and $e_1'$ be two edges subject to \cref{clm:e1e1':s=0}.
The factor-criticality of the component~$C_1$ implies that the subgraph $C_1-V(e_1)$
has a perfect matching, say,~$M_{11}$, in the graph~$H$.
For the same reason, the subgraph $C_1'-V(e_1')$ has a perfect matching, say,~$M_{11}'$, in the graph $H'$.

\begin{clm}\label{clm:M'':s=0}
The graph $F$ has a perfect matching $M''$ such that 
\begin{align}
E_{M_0}(C_1,\,C_2)-M''&\;\neq\;\emptyset,\label[ineq]{cond1:M'':s=0}\rmand\\
E_{M_0'}(C_1',\,C_2')-M''&\;\neq\;\emptyset.\label[ineq]{cond2:M'':s=0}
\end{align}
\end{clm}

We will treat two cases according to whether the equality in \cref{ub:V22} holds or not.
Assume that the equality in \cref{ub:V22} does not hold.
Then the strict inequality in \cref{Hamilton:V22} holds.
By \cref{lem:Plummer},
the subgraph $H[V_{22}]$ is bi-critical. 
In particular, the subgraph $H[V_{22}]-V(e_1)-V(e_1')$ has a perfect matching, say,~$M_{12}$. 
Therefore, the graph~$F$ has the perfect matching 
$M_{11}\cup M_{11}'\cup M_{12}\cup\{e_1,\,e_1'\}$.
See \cref{fig:M1:s=0}.
\begin{figure}[htbp]
\centering
\begin{tikzpicture}
\coordinate (C1) at (-.5, 2.5);
\coordinate (C1') at (5.5, 1);
\coordinate (V22) at (5.5, 4);
\coordinate (L1) at (1, 1);
\coordinate (L2) at (1, 4);
\coordinate (R1) at (3.5, 1);
\coordinate (R2) at (4, 1.5);
\coordinate (R3) at (4, 3.5);
\coordinate (R4) at (3.5, 4);
\coordinate (eVC1) at (2.5, 4);
\coordinate (eVC1') at (4.3, 2.5);

\draw[fill=blue!20] (1, 2.5) ellipse [x radius=1cm, y radius=2.5cm]; %C1
\draw[fill=red!20] (4, 1) circle [radius=1cm]; %C1'
\draw[fill=yellow!20] (4, 4) circle [radius=1cm]; %V22
%\draw[help lines,use as bounding box] (0, 0) grid (12,5);
\draw[very thick]%matching
(L2) -- (R4)
(R2) -- (R3)
;
\draw
(C1)node[]{$C_1$}
(C1')node[]{$C_1'$}
(V22)node[]{$V_{22}$}
(1, 2.5)node[]{$M_{11}$}
(4, .8)node[]{$M_{11}'$}
(4.2, 4.2)node[]{$M_{12}$}
(eVC1)node[above]{$e_1$}
(eVC1')node[]{$e_1'$}
;
\foreach \point in {L2, R2,R3,R4}
\fill[black,opacity=1] (\point) circle (3pt);
\end{tikzpicture}
\caption{The perfect matching $M_{11}\cup M_{11}'\cup M_{12}\cup\{e_1,\,e_1'\}$.}\label{fig:M1:s=0}
\end{figure}
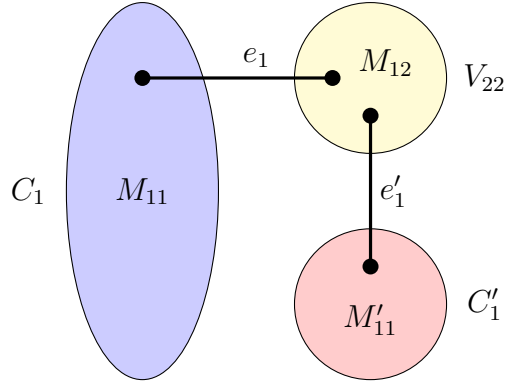

It follows that 
\begin{align}
E_{M_0}(C_1,\,C_2)\cap M_{11}&\=\{e_1\},\label{pf:e1}\rmand\\
E_{M_0'}(C_1',\,V_{22})\cap M_{11}'&\=\{e_1'\}.\label{pf:e1'}
\end{align}
In this case, we define $M''=M_1$.
From \cref{def:M:s=0,pf:e1}, we obtain \cref{cond1:M'':s=0}. It remain to verify \cref{cond2:M'':s=0}.
Recall from \cref{rl:M'} that $E_H(C_1',V_{22})\subseteq M_0'$,
we deduce that 
\[
E_H(C_1',V_{22})\sse E_{M_0'}(C_1',V_{22}).
\]
Together with \cref{eC1'V22}, we infer that 
\[
e_{M_0'}(C_1',V_{22})\gge e_{C_2}(C_1',V_{22})\gge 2.
\]
In view of \cref{pf:e1'}, we infer that $E_{M_0'}(C_1',V_{22})-M_1\ne\emptyset$.
This verifies \cref{cond2:M'':s=0}.

\smallskip
Now, suppose that the equality in \cref{ub:V22} holds.
Then 
\[
|V_{22}|\=\nt-2
\rmand
c_1\=|C_1'|\=\nf+1.
\]
In follows that the number $n/4$ is an integer.
Consider the underlying graph $F$.
On one hand, every vertex has degree at least $n/4+1$. Since $\partial_FC_1\subset M_0$,
we infer that the component~$C_1$ is isomorphic to the complete graph $K_{n/4+1}$,
and that every vertex in $C_1$ sends an edge to the component $C_2$ in the matching~$M_0$.
It follows that 
\begin{equation}\label{complete:C1C2}
e_{M_0}(C_1,\,C_2)\=\nf+1.
\end{equation}
%Along the same lines, the component~$C_1'$ is also isomorphic to $K_{n/4+1}$, and that 
%and that every vertex in~$C_1'$ sends an edge to the component~$C_2'$ in the matching~$M'$.
%It follows that 
%\begin{equation}\label{complete:C1'C2'}
%e_{M'}(C_1',\,C_2')\=\Bcl{\nf}+1.
%\end{equation}

Assume that $E_{M_0}(C_1,C_1')\ne\emptyset$.
Then we can suppose that $e_2\in E_{M_0}(C_1,C_1')$.
Since the component~$C_1$ is factor-critical,
the subgraph \hbox{$F[C_1-V(e_2)]$} has a perfect matching, say, $M_{21}$.
Since the component~$C_1'$ is factor-critical,
the subgraph \hbox{$F[C_1'-V(e_2)]$} has a perfect matching, say, $M_{21}'$.
Let $M_{22}$ be a perfect matching taken from the Hamiltonian cycle $H_{22}$ of the subgraph~$H[V_{22}]$.
Therefore, the graph~$F$ has the perfect matching 
$M_{21}\cup M_{21}'\cup M_{22}\cup \{e_2\}$.
See \cref{fig:M2:s=0}.
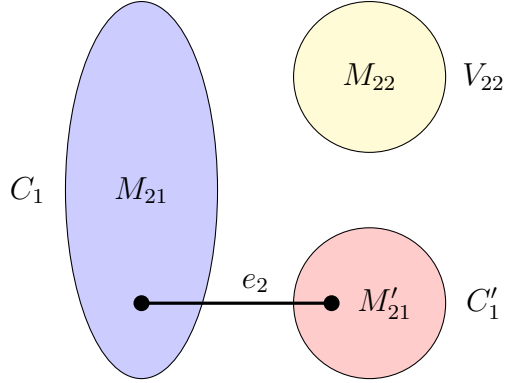
\begin{figure}[htbp]
\centering
\begin{tikzpicture}
\coordinate (C1) at (-.5, 2.5);
\coordinate (C1') at (5.5, 1);
\coordinate (V22) at (5.5, 4);
\coordinate (L1) at (1, 1);
\coordinate (L2) at (1, 4);
\coordinate (R1) at (3.5, 1);
\coordinate (R2) at (4, 1.5);
\coordinate (R3) at (4, 3.5);
\coordinate (R4) at (3.5, 4);
\coordinate (eVC1) at (2.5, 4);
\coordinate (eVC1') at (4.3, 2.5);
\coordinate (eCC) at (2.5, 1);

\draw[fill=blue!20] (1, 2.5) ellipse [x radius=1cm, y radius=2.5cm]; %C1
\draw[fill=red!20] (4, 1) circle [radius=1cm]; %C1'
\draw[fill=yellow!20] (4, 4) circle [radius=1cm]; %V22
%\draw[help lines,use as bounding box] (0, 0) grid (12,5);
\draw[very thick]%matching
(L1) -- (R1)
;
\draw
(C1)node[]{$C_1$}
(C1')node[]{$C_1'$}
(V22)node[]{$V_{22}$}
(1, 2.5)node[]{$M_{21}$}
(4.2, 1)node[]{$M_{21}'$}
(4, 4)node[]{$M_{22}$}
(eCC)node[above]{$e_2$}
;
\foreach \point in {L1, R1}
\fill[black,opacity=1] (\point) circle (3pt);
\end{tikzpicture}
\caption{The perfect matching $M_{21}\cup M_{21}'\cup M_{22}\cup \{e_2\}$.}\label{fig:M2:s=0}
\end{figure}

In this case, we define $M''=M_2$. 
By \cref{def:M:s=0} and the fact $M_2\cap M_0=\{e_2\}$, we verify \cref{cond1:M'':s=0}.
By \cref{eC1'V22} and the fact $M_2\cap M_0'=\emptyset$, we verify \cref{cond2:M'':s=0}.

Otherwise, all edges with one end in the component $C_1$ must have the other end in the set~$V_{22}$.
By \cref{complete:C1C2}, we have $e_{M_0}(C_1,V_{22})\ge n/4+1$.
Recall from \cref{clm:e1e1':s=0} that $e_1'\in E_{M_0'}(C_1',V_{22})$.
With the assumption $|V_{22}|=n/2-2$,
we may choose an edge $e_3\in E_{M_0}(C_1,V_{22})$ such that 
the subgraph $H_{22}-V(e_3)-V(e_1')$ consists of two paths of even orders. 
Consequently, the subgraph $H_{22}-V(e_3)-V(e_1')$ has a perfect matching, say, $M_{32}$.
Since the subgraph $C_1$ is factor-critical,
the subgraph $C_1-V(e_3)$ has a perfect matching, say, $M_{31}$.
Therefore, the graph~$F$ has the perfect matching
$M_{31}\cup M_{11}'\cup M_{32}\cup\{e_3,\,e_1'\}$.
See \cref{fig:M3:s=0}.
\begin{figure}[htbp]
\centering
\begin{tikzpicture}
\coordinate (C1) at (-.5, 2.5);
\coordinate (C1') at (5.5, 1);
\coordinate (V22) at (5.5, 4);
\coordinate (L1) at (1, 1);
\coordinate (L2) at (1, 4);
\coordinate (R1) at (3.5, 1);
\coordinate (R2) at (4, 1.5);
\coordinate (R3) at (4, 3.5);
\coordinate (R4) at (3.5, 4);
\coordinate (eVC1) at (2.5, 4);
\coordinate (eVC1') at (4.3, 2.5);

\draw[fill=blue!20] (1, 2.5) ellipse [x radius=1cm, y radius=2.5cm]; %C1
\draw[fill=red!20] (4, 1) circle [radius=1cm]; %C1'
\draw[fill=yellow!20] (4, 4) circle [radius=1cm]; %V22
%\draw[help lines,use as bounding box] (0, 0) grid (12,5);
\draw[very thick]%matching
(L2) -- (R4)
(R2) -- (R3)
;
\draw
(C1)node[]{$C_1$}
(C1')node[]{$C_1'$}
(V22)node[]{$V_{22}$}
(1, 2.5)node[]{$M_{31}$}
(4, .8)node[]{$M_{11}'$}
(4.2, 4.2)node[]{$M_{32}$}
(eVC1)node[above]{$e_3$}
(eVC1')node[]{$e_1'$}
;
\foreach \point in {L2, R2,R3,R4}
\fill[black,opacity=1] (\point) circle (3pt);
\end{tikzpicture}
\caption{The perfect matching $M_{31}\cup M_{11}'\cup M_{32}\cup\{e_3,\,e_1'\}$.}\label{fig:M3:s=0}
\end{figure}
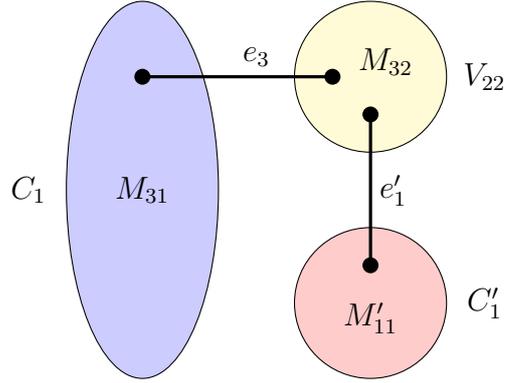

In this case, we define $M''=M_3$. 
By \cref{def:M:s=0} and the fact $M_3\cap M_0=\{e_3\}$, we verify \cref{cond1:M'':s=0}.
By \cref{eC1'V22} and the fact $M_3\cap M_0'=\{e_1'\}$, we verify \cref{cond2:M'':s=0}.
This proves \cref{clm:M'':s=0}.
\qed
\medskip

Let $M''$ be a perfect matching of the graph $F$ chosen subject to \cref{cond1:M'':s=0,cond2:M'':s=0}.
By \cref{clm:CC:s=0}, we can suppose that 
the graph $H''=F-M''$ consists of the factor-critical components~$C_1''$ and~$C_2''$ such that 
\begin{equation}\label[ineq]{lb:c1'':s=0}
\Bcl{\nf}+1\lle |C_1''|\lle |C_2''|.
\end{equation}

\begin{clm}\label{clm:V22:s=0}
We have $V(C_1'')\subseteq V_{22}$.
\end{clm}

By \cref{lem:CC:s=0,cond1:M'':s=0}, we obtain
\begin{equation}\label[rl]{C1''subsetC2}
V(C_1'')\;\subset\; V(C_2).
\end{equation}
On the other hand,
we apply \cref{lem:CC:s=0} by replacing the triple $(H,\,M,\,M')$ in its statement by the triple $(H',\,M_0',\,M'')$.
Let us check the conditions of \cref{lem:CC:s=0} one by one.
First, from the definition $H'=(H\cup M_0)-M_0'$, 
the graph $H'$ has minimum degree $\delta(H)\ge\cl{n/4}$,
consists of factor-critical components~$C_1'$ and~$C_2'$ with $|C_1'|\le |C_2'|$,
and has no intersection with the perfect matching $M_0'$.
Second, from definition, the graph 
\[
(H'\cup M_0')-M''\=F-M''\=H''
\]
consists of factor-critical components $C_1''$ and $C_2''$ with $|C_1''|\le |C_2''|$.
Therefore, by \cref{lem:CC:s=0,cond2:M'':s=0}, we obtain 
\begin{equation}\label[rl]{C1''subsetC2'}
V(C_1'')\;\subset\; V(C_2').
\end{equation}
Combining \cref{C1''subsetC2,C1''subsetC2'}, we find
\[
V(C_1'')\;\subseteq\; V(C_2)\cap V(C_2')\=V_{22}.
\]
This proves \cref{clm:V22:s=0}.
\qed
\medskip

By \cref{clm:V22:s=0}, the vertex set $V_{22}$ is decomposed into two parts as
\[
V_{22}\=V(C_1'') \sqcup W,
\]
where the vertex set $W$ is defined by the above decomposition.
Note that all the orders $c_2$, $|C_1'|$, and $|C_1''|$ are odd. From definition, we find the order 
\[
|W|\=c_2-|C_1'|-|C_1''|
\] 
is odd, which implies that $W\ne\emptyset$.
By \cref{rl:M'}, we have
\[
E_H(W,\,C_1')\sse \partial_{C_2}C_1'\sse M_0'.
\]
%\begin{equation}\label[rl]{W'}
%E_H(W,\,C_1')\sse \partial_{C_2}C_1'\sse M'.
%\end{equation}
Similarly, we have
\[
E_H(W,\,C_1'')\sse \partial_{C_2}C_1''\sse M''.
\]
%\begin{equation}\label[rl]{W''}
%E_H(W,\,C_1'')\sse \partial_{C_2}C_1''\sse M''.
%\end{equation}
By the above two relations, we find that every vertex in the set $W$ has at most two neighbors outside $W$ in the component $C_2$.
By \cref{range:deg}, every vertex in $W$ has degree at least $\cl{n/4}-2$.
It follows that $|W|\gge \cl{n/4}-1$.
By \cref{lb:c1'':s=0}, we infer that 
\begin{equation}\label[ineq]{pf:V22:s=0}
|V_{22}|\=|C_1''|+|W|\gge \Bg{\nf+1}+\Bg{\nf-1}\=\nt,
\end{equation}
contradicting \cref{ub:V22}. 
\end{case}

This completes the proof of \cref{thm:main}.
\end{proof}

The sharpness of the number $D_n$ in \cref{thm:main} 
can be seen from the $(D_n-1)$-regular graph without perfect matchings
pointed out by Csaba et al.~\cite{CKLOT14X}.
In fact,
when the integer~$n/2$ is odd, consider the disjoint union of two cliques of order $n/2$;
when~$n/2$ is even, consider the graph obtained from the disjoint union of cliques 
of orders $(n/2-1)$ and $(n/2+1)$ by deleting a Hamiltonian cycle in the larger clique.

The sharpness of the bound $\cl{n/4}$ in \cref{thm:main} 
can be seen in the sense of \cref{thm:sharp1}.

\begin{thm}\label{thm:sharp1}
Let $n\ge 34$ be an even integer.
There exists a $\{D_n,\,D_n+1\}$-graph of order $n$ having exactly $\cl{n/4}$ disjoint perfect matchings. 
\end{thm}

\begin{proof}
Let $n\ge 34$ and denote $D=D_n$.
By \cref{thm:main}, 
it suffices to construct a $\{D,D+1\}$-graph of order $n$ having at most $\cl{n/4}$ disjoint perfect matchings. 
Let $K$ be the complete bipartite graph with part orders \hbox{$|A|=n/2-1$} and $|B|=n/2+1$. 

Suppose that the integer $n/2$ is odd. Then we have $D=n/2$ from \cref{def:D}.
Define~$G_1$ to be the graph obtained from the graph~$K$ by adding 
a perfect matching~$M_1$ that covers the vertex set~$V(B)$.
%From definition, 
%any vertex in the part~$A$ has degree $(n/2+1)$,
%and that every vertex in the part~$B$ has degree~$n/2$.
Then the graph~$G_1$ is a $\{D,D+1\}$-graph of order~$n$.
It is clear that every matching of~$G_1$ contains exactly one edge in the subgraph~$G_1[B]$.
Hence, the cardinality of the maximum family of disjoint perfect matchings of the graph~$G_1$ is at most $|M_1|=n/4$.
In this case, the graph~$G_1$ is a desired graph.

Otherwise, the integer $n/2$ is even and $D=n/2-1$.
Let~$M$ be a maximal matching of the graph~$K$. 
Define~$G_2$ to be the graph obtained from the graph $K-M$ by adding 
a minimal edge set~$E_2$ that covers the vertex set $V(M)-V(A)$.
%From definition, the degree of any vertex in the part~$A$ is $n/2$,
%and that the degree of any vertex in the part~$B$ is either~$n/2-1$ or~$n/2$.
Then the graph~$G_2$ is a $\{D,D+1\}$-graph of order~$n$.
It is clear that every matching of~$G_2$ contains exactly one edge in the subgraph~$G_2[B]$.
Hence, the number of disjoint perfect matchings of $G_2$ is at most
\[
|E_2|\=\bggcl{\frac{|V(M)|-V(A)}{2}}\=\Bcl{\frac{n/2-1}{2}}\=\Bcl{\nf}.
\]
In this case, the graph $G_2$ is qualified.
This completes the proof.
\end{proof}

\begin{cor}\label{cor:main}
Let $n$ be an even integer, and let $D\ge D_n$.
Then every $\{D,\,D+1\}$-graph of order $n$ contains $\cl{(D+1)/2}$ disjoint perfect matchings.
\end{cor}

\begin{proof}
Let $G$ be a $\{D,\,D+1\}$-graph of order $n$.
If $n=2$, then $G$ is isomorphic to the complete graph of order two, which has a perfect matching certainly.
Otherwise $n\ge 4$.
If $D>D_n$, then the minimum degree 
\[
\delta(G)\=D\gge D_n+1\=2\Bcl{\nf}\gge \nt.
\]
By Dirac's \cref{thm:Dirac}, the graph $G$ is Hamiltonian, and thus has a perfect matching, say, $M_1$.
Now, consider the graph $G_1=G-M_1$. It is clear that the graph $G_1$ is \hbox{$\{D-1,\,D\}$}-regular.
If $D-1>D_n$, then we can choose a perfect matching $M_2$ from the graph $G_1$ for the same reason.
Continuing in this way, we obtain disjoint perfect matchings $M_1,M_2,\ldots,M_{D-D_n}$, and the 
$\{D_n,\,D_n+1\}$-graph
\[
G_{D-D_n}\=G-M_1-M_2-\cdots-M_{D-D_n}.
\] 
By \cref{thm:main}, the graph $G_{D-D_n}$ has a family $\M$ of $\cl{n/4}$ disjoint perfect matchings. 
Hence, the graph $G$ has the family $\M\cup M_1\cup M_2\cup\cdots\cup M_{D-D_n}$ of 
\[
D-D_n+\Bcl{\nf}\=D-\Bcl{\nf}+1\gge \Bcl{\frac{D+1}{2}}
\]
disjoint perfect matchings.
\end{proof}

\section{Concluding remarks}

Note that semi-regular graphs are certainly general graphs,
for which Csaba et al.~\cite{CKLOT14X} also presented a sharp bound 
for the maximum number of disjoint perfect matchings.

\begin{thm}[Csaba et al.]\label{thm:CKLOT2}
For sufficiently large even integer $n$,
any graph of order~$n$ with minimum degree at least $n/2$ 
contains at least $(n-2)/8$ disjoint Hamiltonian cycles.
\end{thm}

We point out that \cref{thm:CKLOT2} has intersection with our \cref{thm:main},
and that none of them covers the other. The differences include the following.
\begin{itemize}
\item
\cref{thm:main} involves the case $D_n=n/2-1$, while \cref{thm:CKLOT2} does not.
In particular, the bound $n/2$ for the minimum degree in \cref{thm:CKLOT2} is sharp; while in our result,
$\{n/2-1,\,n/2\}$-graphs has minimum degree $n/2-1$.
\item
For $D_n=n/2$, \cref{thm:main} says every $\{D_n,\,D_n+1\}$-graph contains $\cl{n/4}=(n+2)/4$
disjoint perfect matchings, while \cref{thm:CKLOT2} implies only $2\cdot (n-2)/8=(n-2)/4$ disjoint perfect matchings;
\item 
\cref{thm:main} holds true for all even integers $n\ge 34$, while \cref{thm:CKLOT2} is valid for sufficient large $n$.
\end{itemize}

%\begin{cor}
%Let $G$ be a semi-regular graph of even order $n\ge 34$.
%Suppose that neither the graph $G$ nor its complement contains $\cl{n/4}$ disjoint perfect matchings,
%then $n\equiv 2\pmod4$ and the graph $G$ is a $\{n/2-1,\,n/2\}$-graph.
%\DWr{Is it possible to determine the structure of such graphs $G$?}
%\end{cor}
%
%\begin{proof}
%Let $G$ be a $\{k,\,k+1\}$-graph of even order $n\ge 34$. 
%Then the complement graph~$G^c$ is a $\{n-k-2,\,n-k-1\}$-graph.
%From the premise and \cref{thm:main}, we can suppose that both 
%the numbers~$k$ and $(n-k-2)$ are at most $D_n-1$.
%Then, it is elementary to derive that $k=n-k-2=D_n-1$,
%which implies that $n\equiv 2\pmod4$ and $k=n/2-1$.
%This completes the proof.
%\end{proof}
%

\end{document}